\documentclass[11pt]{article}
\usepackage[english]{babel}
\usepackage[utf8x]{inputenc}
\usepackage[T1]{fontenc}
\usepackage{a4wide}
\usepackage{graphicx}
\usepackage{mathrsfs, graphicx}
\usepackage{amsmath, amsfonts, stmaryrd, amssymb, dsfont}
\usepackage{amsthm}
\usepackage{array}
\usepackage{multirow}
\usepackage[left=2.5cm, right=2.5cm,top=3cm,bottom=3cm]{geometry}
\usepackage{subfig}
\usepackage{float, caption}
\usepackage{color}

\newcommand{\N}{\mathbb{N}}                                              
\newcommand{\Z}{\mathbb{Z}}

\newcommand{\R}{\mathbb{R}}                                              

\newcommand{\T}{\mathbb{T}}

\renewcommand{\P}{\mathbb{P}}
\newcommand{\ds}{\displaystyle}

\newcommand{\A}{\mathcal{A}}
\newcommand{\p}{\mathbf{p}}
\renewcommand{\r}{\mathbf{r}}
\renewcommand{\S}{\mathcal{S}}
\renewcommand{\L}{\mathcal{L}}
\renewcommand{\d}{\partial}
\newtheorem{de}{Definition}
\newtheorem{theo}{Theorem}
\newtheorem{ex}{Example}
\newtheorem{lem}{Lemma}
\newtheorem{prop}{Proposition}
\newtheorem{cor}{Corollary}

\numberwithin{theo}{section}
\numberwithin{lem}{section}
\numberwithin{ex}{section}
\numberwithin{ex}{section}
\numberwithin{equation}{section}
\numberwithin{prop}{section}
\numberwithin{cor}{section}
\begin{document}

\author{Marielle Simon\footnote{UMPA, UMR-CNRS 5669, ENS de Lyon, 46 all\'ee d'Italie, 69007 Lyon, France. \textit{Mail}: marielle.simon@ens-lyon.fr}}
\title{Hydrodynamic Limit for the Velocity-Flip Model}

\maketitle

\begin{abstract}
We study the diffusive scaling limit for a chain of $N$ coupled oscillators. In order to provide the system with good ergodic properties, we perturb the Hamiltonian dynamics with random flips of velocities, so that the energy is locally conserved. We derive the hydrodynamic equations by estimating the relative entropy with respect to the local equilibrium state, modified by a correction term. 
\end{abstract}

\thanks{\textsc{Acknowledgements.} { I thank C\'edric Bernardin and Stefano Olla for giving me this problem, and for the useful discussions and suggestions on this work.}}

\section*{Introduction}

This paper aims at proving the hydrodynamic limit for a Hamiltonian system of $N$ coupled oscillators. The ergodic properties of Hamiltonian dynamics are poorly understood, especially when the size of the system goes to infinity. That is why we perturb it by an additional conservative mixing noise, as it has been proposed for the first time by Olla, Varadhan and Yau \cite{MR1231642} in the context of gas dynamics, and then in \cite{MR1278883} in the context of Hamiltonian lattice dynamics (see e.g. \cite{BBO},  \cite{MR2305383}, \cite{BS12}, \cite{BerKan}, \cite{MR2185330}, \cite{BerOlla}, \cite{EveOll}, \cite{LO} for more recent related works).

We are interested in the macroscopic behavior of this system as $N$ goes to infinity, after rescaling space and time with the diffusive scaling. The system is considered under periodic boundary conditions -- {more precisely we work on the one-dimensional discrete torus $\T_N:=\{0,...,N-1\}$.  The configuration space is denoted by $\Omega_N:=\left(\R\times\R\right)^{\T_N}$. A typical configuration is given by $\omega=(p_x,r_x)_{x \in \T_N}$ where $p_x$ stands for the velocity of the oscillator at site $x$, and $r_x$ represents the distance between oscillator $x$ and oscillator $x+1$. The deterministic dynamics is described by the harmonic Hamiltonian}
\begin{equation} \mathcal{H}_N=\sum_{x=0}^{N-1}\left[ \frac{p_x^2}{2}+\frac{r_x^2}{2}\right]\ . \end{equation}
The stochastic perturbation is added only to the velocities, in such a way that the energy of particles is still conserved. Nevertheless, the momentum conservation is no longer valid. The added noise can be easily described: each particle independently waits an exponentially distributed time interval and then flips the sign of velocity. {The strength of the noise is regulated by the parameter $\gamma >0$}. The total deformation $\sum r_x$  and the total energy $\sum (p_x^2 + r_x^2)/2$ are the { only two} conserved quantities. Thus, the Gibbs states are parametrized by two potentials, temperature and tension: for $\beta >0$ and $\lambda \in \R$, the equilibrium Gibbs measures $\mu_{\beta,\lambda}^N$ on the configuration space $\Omega^N:=(\R\times \R)^{\T_N}$ are given by the product measures\begin{equation} d \mu_{\beta,\lambda}^N=\prod_{x \in \T_N} \frac{e^{-\beta e_x-\lambda  r_x} }{Z(\beta,\lambda)} dr_x dp_x \ , \end{equation} where $ e_x:=(p_x^2+r_x^2)/2$ is the energy of the particle at site $x$, and $Z(\beta,\lambda)$ is the normalization constant.  The temperature is equal to $\beta^{-1}$ and the tension is given by $\lambda / \beta$.

The goal is to prove that the two empirical profiles associated to the conserved quantities converge in the thermodynamic limit $N \to \infty$ to the macroscopic profiles $\mathbf{r}(t,\cdot)$ and $\mathbf{e}(t,\cdot)$ which satisfy an autonomous system of coupled parabolic equations. Let $\mathbf{r}_0 : \T \to \R$ and $\mathbf{e}_0 : \T \to \R$ be respectively the initial macroscopic deformation profile and the initial macroscopic energy profile defined on the one-dimensional torus $\T=[0,1]$. We want to show that the functions $\mathbf{r}(t ,q)$ and $ \mathbf{e}(t,q)$ defined on $\R_+ \times \T$ are solutions of  \begin{equation} \left\{\begin{aligned} \d_t \mathbf{r} &=\frac{1}{\gamma} \d^2_q \mathbf{r} \ , \\
 \d_t \mathbf{e} & = \frac{1}{2\gamma} \d^2_q\left(\mathbf{e}+\frac{\mathbf{r}^2}{2}\right)\ ,\end{aligned} \right.  \quad  q \in \T, \ t \in \R\ , \label{h} \end{equation} with the initial conditions $\mathbf{r}(0,\cdot)=\mathbf{r}_0(\cdot)$ and $\mathbf{e}(0,\cdot)=\mathbf{e}_0(\cdot)$.

We approach this problem by using the relative entropy method, introduced for the first time by H. T. Yau \cite{MR1121850}  for a gradient {\footnote{A conservative system is called gradient if the currents corresponding to the conserved quantities are gradients.}} diffusive Ginzburg-Landau dynamics. For non-gradient models, Varadhan  \cite{MR1354152} has proposed an effective approach. Funaki et al. followed his ideas in \cite{MR1395890} to extend the relative entropy method to some non-gradient processes and introduced the concept of local equilibrium state of second order approximation. 

The usual relative entropy method works with two time-dependent probability measures. Let us denote by $\mu_0^N$ the Gibbs local equilibrium associated to a deformation profile $\mathbf{r}_0$ and an energy profile $\mathbf{e}_0$  (see \eqref{gibbs} for the explicit formula).  As we  work in the diffusive scaling, we look at  the state { of the process at time $tN^2$.  We denote it by $\mu_t^N$ and we suppose that it starts from $\mu_0^N$}. Let  $\mu_{\mathbf{e}(t,\cdot),\mathbf{r}(t,\cdot)}^N$ be the Gibbs local equilibrium  associated to the profiles $\mathbf{r}(t,\cdot)$ and $\mathbf{e}(t,\cdot)$ which satisfy \eqref{h}\ \footnote{For the sake of readibility, in the following sections we will denote it by $\mu^N_{\beta_t(\cdot),\lambda_t(\cdot)}$, where $\beta_t(\cdot)$ and $\lambda_t(\cdot)$ are the two potential profiles associated to $\mathbf{r}(t,\cdot)$ and $\mathbf{e}(t,\cdot)$ (see  \eqref{rel2} and \eqref{gibbs}).}. If we denote by $f_t^N$ and $\phi_t^N$, respectively, the densities\footnote{The existence of these two densities is justified in Section \ref{subsec:method}.} of $\mu_t^N$ and $\mu_{\mathbf{e}(t,\cdot),\mathbf{r}(t,\cdot)}^N$  with respect to a reference equilibrium measure $\mu_*^N:=\mu_{1,0}^N$, we guess that $\phi_t^N$ is a good approximation of the unknown density $f_t^N$. We measure the distance between these two densities by their relative entropy  \begin{equation} H_N(t):=\int_{\Omega^N} f_t^N(\omega) \log \frac{f_t^N(\omega)}{\phi_t^N(\omega)}\ d\mu_*^N(\omega)\ . \end{equation}
Then, the strategy consists in proving that \begin{equation} \lim_{N \to \infty} \frac{H_N(t)}{N} = 0 \ , \label{eq:rel} \end{equation} and deducing that the hydrodynamic limit holds (for this last step, see  \cite{MR1707314}, \cite{MR1231642} or \cite{BOmanus}). In the context of diffusive systems, the relative entropy method works if the following conditions are satisfied. \begin{itemize}
\item First, the dynamics has to be \textit{ergodic}: the only time and space invariant measures for the infinite system, with finite local entropy, are given by mixtures of the Gibbs measures in infinite volume $\mu_{\beta,\lambda}$ (see \eqref{eq:infinite}).  From \cite{MR1278883}, we know that the velocity-flip model is ergodic in the sense above (see Theorem \ref{theo:ergod}   for a precise statement). 
\item Next, we need to establish the so-called \textit{fluctuation-dissipation equations} in the mathematics literature. Such equations express the microscopic current of energy (which here is not a discrete gradient) as the sum of a discrete gradient and a fluctuating term. { More precisely, the microscopic current of energy, denoted by $j_{x,x+1}$, is defined by the local energy conservation law}
\begin{equation} \L e_x=\nabla j_{x-1,x} \end{equation}
{where $\L$ is the generator of the infinite dynamics. The standard approach consists in proving that there exist functions $f_x$ and $h_x$ such that the following decomposition holds } \begin{equation}  j_{x,x+1} = \nabla f_x + \L h_x\ . \label{fluct} \end{equation}
{Equation \eqref{fluct} is called a microscopic fluctuation-dissipation equation. The term $\L h_x$, when integrated in time, is a martingale. Roughly speaking, $\L h_x$ represents rapid fluctuation, whereas $\nabla f_x$ represents dissipation. Gradient models are systems for which $h_x = 0$ with the previous notations.} In general, these equations are not explicit but we are able to compute them in our model (see \eqref{diff1} and \eqref{diff2}).
\item Finally, since we observe the system on a diffusive scale and the system is non-gradient, we need second order approximations. If we want to obtain the entropy estimate \eqref{eq:rel} of order $o(N)$, we can not work with the measure $\mu_{\mathbf{e}(t,\cdot),\mathbf{r}(t,\cdot)}^N$:  we have to correct the Gibbs local equilibrium state with a small term. This idea was first introduced in \cite{MR1395890} and then used in  \cite{MR1934159} for interacting Ornstein-Uhlenbeck processes,  and in \cite{MR2080952}  for the asymmetric exclusion process in the diffusive scaling. However, as far as we know, it is the first time that this is applied for a system with several conservation laws.

\end{itemize}

Recently, Even et al.   \cite{EveOll} used  the relative entropy method for a stochastically perturbed Hamiltonian dynamics which is quite close to the dynamics of this paper: the time evolution is governed by the same Hamiltonian of anharmonic oscillators but the process is perturbed by a different noise -- velocities are exchanged and not flipped. Besides, the boundary conditions are mechanical instead of periodic. Contrary to this paper, the model is studied in the hyperbolic scale, so that the authors do not need to modify the local equilibrium state. 

Up to present, the derivation of hydrodynamic equations for the harmonic oscillators perturbed by the velocity-flip noise is not rigorously achieved (see e.g. \cite{BerKan}), mainly because the control of large energies has not been considered so far. Indeed, to perform the relative entropy method, we need to control the moments  
\begin{equation}
\label{eq:mombound}
\int \left[\frac 1N \sum_{x\in \T_N} |p_x|^k\right] d\mu_t^N\ , 
\end{equation}
for all $k\geqslant 1$, uniformly in time and with respect to $N$. In fact, the { only first several moments} are necessary to cut-off large energies (as it is explained in Section \ref{subsec:cutoff}) and we need all the others to obtain the Taylor expansion which appears in the relative entropy method (see Proposition \ref{prop} and Lemma \ref{lem:sym}).
Usually, the entropy inequality \eqref{entropin} reduces the control of \eqref{eq:mombound} to the estimate of the following equilibrium exponential moments
$$\int \exp(\delta |p_x|^k) \ d\mu_{1, 0}^N$$ 
with $\delta>0$ small. Unfortunately, in our model, these integrals are infinite for all $k \geqslant 3$ and all $\delta>0$. 

To avoid this problem, we could cut-off large velocities by taking a relativistic kinetic energy (as done in \cite{MR1231642}). { Nevertheless, we should change the physics of the problem by modifying the Liouville operator, and consequently the fluctuation-dissipation equations would not be available any more.} Similar difficulties have already appeared in other models: in \cite{MR2192537}, Bertini et al. do not have these precious exponential moments to derive rigorously their results. In an other context, Bonetto et al. \cite{MR2518984} study the heat conduction in anharmonic crystals with self-consistent reservoirs, and need energy bounds to complete their results. Bernardin   \cite{MR2305383} deals with {a} harmonic chain perturbed by a stochastic noise which is different from ours but has the same motivation: energy is conserved, momentum is not. He derives the hydrodynamic limit for a particular value of the intensity of the noise. In this case the hydrodynamic equations are simply given by two decoupled heat equations. The author highlights that good energy bounds are necessary to extend his work to other values of the noise intensity. In fact, only the following weak form is proved in his paper:
\begin{equation} \lim_{N \to + \infty} \int \left[\frac{1}{N^2} \sum_{x\in \T_N} p_x^4\right] d\mu_t^N =0\ . \end{equation}
In this work, we get uniform control of (\ref{eq:mombound}) (Theorem \ref{theo:moments}), thanks to a remarkable property of our model: the set of convex combinations of Gaussian measures is preserved by the dynamics. This is one of the main technical novelties in our work.

The next section contains a more precise description of the results outlined here, along with the plan of the paper.

\section{The Model and the Main Results}\label{sec:model}
\subsection{Velocity-flip Model}\label{subsec:model}

We consider the unpinned harmonic chain perturbed by the momentum-flip noise. Each particle has the same mass that we set equal to 1. The configuration space is denoted by $\Omega^N:=(\R\times \R)^{\T_N}.$ 

A typical configuration is $\omega=(\r,\p)\in \Omega^N$, where $\r=(r_x)_{x \in \T_N}$ and $\p=(p_x)_{x \in \T_N}$. 

The generator of the dynamics is given by $\L_N:=\A_N+ \gamma  \S_N$, where, for any continuously differentiable function $f : \Omega^N \to \R$, \begin{equation}\A_N(f) := \sum_{x \in \T_N} [(p_{x+1}-p_x) \ \d_{r_x}f + (r_x-r_{x-1}) \ \d_{p_x}f]\end{equation}
and  \begin{equation} \S_N(f)(\r,\p):=\frac{1}{2} \sum_{x \in \T_N} [f(\r,\p^x)-f(\r,\p)]\ . \end{equation} Here $\p^x$ is the configuration obtained from $\p$ by the flip of $p_x$ into $-p_x$. The parameter $\gamma >0$ regulates the strength of the random flip of momenta. 

The operator $\A_N$ is the Liouville operator of a chain of interacting harmonic oscillators, and $\S_N$ is the generator of the stochastic part of the dynamics that flips at random time the velocity of one particle. The dynamics conserves two quantities: the total deformation of the lattice $ \mathcal{R}=\sum_{x \in \T_N} r_x$ and the total energy $ \mathcal{E}=\sum_{x \in \T_N} e_x$, where $ e_x=(p_x^2+r_x^2)/2$. Observe that the total momentum is no longer conserved.
 
The deformation and the energy define a family of invariant measures depending on two parameters.  For $\beta >0$ and $\lambda \in \R$, we denote by $\mu_{\beta,\lambda}^N$ the Gaussian product measure on $\Omega^N$ given by \begin{equation} \mu_{\beta,\lambda}^N(d\r,d\p)=\prod_{x \in \T_N} \frac{e^{-\beta e_x-\lambda  r_x} }{Z(\beta,\lambda)} dr_x dp_x\ . \end{equation} An easy computation gives that the partition function satisfies \begin{equation} Z(\beta,\lambda)=\frac{2\pi}{\beta} \exp\left(\frac{\lambda^2}{2 \beta}\right) \ . \end{equation}
 
In the following, we shall denote by $\mu[\cdot ]$ the expectation with respect to the measure $\mu$. We introduce $\mathbb{L}^2(\mu_{\beta,\lambda}^N)$, the space of functions $f$ defined on  $\Omega^N$ such that $\mu_{\beta,\lambda}^N[f^2]<+\infty$. This is a Hilbert space, on which $\A_N$ is antisymmetric and $\S_N$ is symmetric.
 
The thermodynamic relations between the averages of the conserved quantities $\bar{\mathbf{r}} \in \R$ and $\bar{\mathbf{e}} \in {(0,+\infty)}$, and the potentials $\beta \in (0,+\infty)$ and $\lambda \in \R$ are given by \begin{equation} \left\{\begin{aligned} \bar{\mathbf{e}}(\beta,\lambda):= \mu_{\beta,\lambda}^N[e_x]& = \frac{1}{\beta} + \frac{\lambda^2}{2\beta^2}\ , \\
\bar{\mathbf{r}}(\beta,\lambda):=\mu_{\beta,\lambda}^N[r_x] & =-\frac{\lambda}{\beta}\ . \end{aligned} \right. \label{rel2} \end{equation}
Let us notice that \begin{equation} \forall \ \beta \in (0,+\infty), \forall \ \lambda \in \R, \ \bar{\mathbf{e}}(\beta,\lambda) > \frac{\bar{\mathbf{r}}^2(\beta,\lambda)}{2} \ . \end{equation}
We assume that  the system is initially close to a local equilibrium (defined as the following).
 
 \begin{de} \label{de:loceq}
 A sequence $(\mu^N)_N$ of probability measures on $\Omega^N$ is a \emph{local equilibrium} associated to a deformation profile $\mathbf{r}_0 : \T \to \R$ and an energy profile $\mathbf{e}_0 : \T \to (0,+\infty)$ if for every continuous function $G : \T \to \R$ and for every $\delta >0$, we have \begin{equation} \left\{ \begin{aligned}\lim_{N \to \infty} \mu^N\left[ \left\vert \frac{1}{N}\sum_{x \in \T_N}G\left(\frac{x}{N}\right)r_x-\int_{\T} G({ q})\mathbf{r}_0({ q})d{ q}\right\vert > \delta \right] &=0\ , \\
 \lim_{N \to \infty} \mu^N\left[ \left\vert \frac{1}{N}\sum_{x \in \T_N}G\left(\frac{x}{N}\right)e_x-\int_{\T} G({ q})\mathbf{e}_0({ q})d{ q}\right\vert > \delta \right] &=0\ .\end{aligned} \right. \end{equation}
 \end{de}
 
 \begin{ex}\label{ex:loceq}
For any integer $N$ we define the probability measures \begin{equation} \mu_{\beta_0(\cdot),\lambda_0(\cdot)}^N(d\r,d\p)=\prod_{x \in \T_N} \frac{\exp(-\beta_0(x/N) e_x-\lambda_0(x/N)  r_x)}{Z(\beta_0(\cdot),\lambda_0(\cdot))} dr_x dp_x\ , \label{gibbs} \end{equation}  

where $\beta_0(\cdot)$ and $\lambda_0(\cdot)$ are related to $\mathbf{e}_0(\cdot)$ and $\mathbf{r}_0(\cdot)$ by  \eqref{rel2} $$\left\{\begin{aligned} \mathbf{e_0}(\cdot)& =\mathbf{\bar{e}}(\beta_0(\cdot),\lambda_0(\cdot))\ , \\ 
\mathbf{r_0}(\cdot) & = \mathbf{\bar{r}}(\beta_0(\cdot),\lambda_0(\cdot))\ .  \end{aligned}\right.  $$

Then, the sequence $\left(\mu^N_{\beta_0(\cdot),\lambda_0(\cdot)}\right)_N$ is a local equilibrium, { and it is called the \emph{Gibbs local equilibrium state} associated to the macroscopic profiles $\mathbf{e}_0$ and $\mathbf{r}_0$}. Both profiles are assumed to be continuous.
 \end{ex}
 
To establish the hydrodynamic limit corresponding to the two conservation laws, we look at the process with generator $N^2 \L_N$, namely in the diffusive scale. The configuration at time $tN^2$ is denoted by $\omega_t^N$, and the law of the process $(\omega_t^N)_{t \geqslant 0}$ is denoted by $\mu_t^N$. 
 
 \subsection{The Thermodynamic Entropy}\label{subsec:entropy}

The function \begin{equation} S(\mathbf{e},\mathbf{r})=\inf_{\beta>0, \lambda\in \R} \left\{ \lambda \mathbf{r} + \beta \mathbf{e} + \log Z(\beta,\lambda) \right\} \end{equation} is called the \textit{thermodynamic entropy}. An easy computation, coming from the explicit expression of the partition function, gives \begin{equation} S(\mathbf{e},\mathbf{r})=1 +\log(2\pi) + \log\left(\mathbf{e}-\frac{\mathbf{r}^2}{2}\right), \ \text{ when } \mathbf{e}-\frac{\mathbf{r}^2}{2} >0\ . \end{equation} The relations \eqref{rel2} can be inverted according to \begin{equation} \lambda(\mathbf{e},\mathbf{r})=\frac{\d S(\mathbf{e},\mathbf{r})}{\d \mathbf{r}}\ , \qquad \beta(\mathbf{e},\mathbf{r})=\frac{\d S(\mathbf{e},\mathbf{r})}{\d \mathbf{e}}\ . \end{equation} \textbf{Remark}. These two equalities, together with \eqref{rel2}, show that there exists a bijection between the two sets $\left\{(\beta, \lambda) \in \R^2 \ ; \ \beta >0 \right\}$ and $\left\{ (\mathbf{e},\mathbf{r}) \in \R^2 \ ; \ \mathbf{e} > \mathbf{r}^2/2 \right\}$. From the equations above, the inverted relations can be written as \begin{equation} \lambda(\mathbf{e},\mathbf{r})=-\frac{\mathbf{r}}{\mathbf{e}-\mathbf{r}^2/2}\ ,\qquad \beta(\mathbf{e},\mathbf{r})=\frac{1}{\mathbf{e}-\mathbf{r}^2/2}\ . \end{equation} {We denote by $\Psi$ the function }$$ \begin{array}{c c c c} \Psi :  & \left\{ (\mathbf{e},\mathbf{r}) \in \R^2 \ ; \ \mathbf{e} > \mathbf{r}^2/2 \right\} & \to & \left\{(\beta, \lambda) \in \R^2 \ ; \ \beta >0 \right\}\\ & (\mathbf{e},\mathbf{r})& \mapsto &\left( \ds  \frac{1}{\mathbf{e}-\mathbf{r}^2/2},\ -\frac{\mathbf{r}}{\mathbf{e}-\mathbf{r}^2/2} \right)\ .\end{array}$$

If ${\eta}=(\mathbf{e},\mathbf{r})$ and $\chi=(\beta,\lambda)$ satisfy the relations \eqref{rel2}, then  ${\eta}$ and $\chi$ are said \textit{in duality} and we have \begin{equation} -S(\mathbf{e},\mathbf{r})+\log Z(\beta,\lambda)={ -} \eta \cdot \chi\ . \label{fenchel} \end{equation}

Here, the notation $a \cdot b $ stands for the usual scalar product between $a$ and $b$.

\subsection{Hydrodynamic Equations}{\label{subsec:equations}

Let $\mu$ and $\nu$ be two probability measures on the same measurable space $(X, \mathcal{F})$. We define the relative entropy $H(\mu \vert \nu)$ of the probability measure $\mu$ with respect to the probability measure $\nu$ by \begin{equation} H(\mu \vert \nu)=\sup_{f} \left\{ \int_X f \ d\mu - \log \left(\int_X e^f \ d\nu\right) \right\}\ , \end{equation} where the supremum is carried over all bounded measurable functions $f$ on $X$.

The Gibbs states in infinite volume are the probability measures $\mu_{\beta,\lambda}$ on $\Omega=(\R\times\R)^\Z$ given by \begin{equation} \mu_{\beta,\lambda}(d\r,d\p)=\prod_{x \in \Z} \frac{e^{-\beta e_x -\lambda r_x}}{Z(\beta,\lambda)}dr_x dp_x\ . \label{eq:infinite} \end{equation}
We denote by $\tau_x \varphi$ the shift of $\varphi$: $(\tau_x \varphi)(\omega)=\varphi(\tau_x \omega)=\varphi(\omega(x+\cdot\ ))$. In this article the following theorem  is proved.
 \begin{theo}\label{theo:hydro} Let $(\mu_0^N)_N$ be a sequence of probability measures on $\Omega^N$ which is a local equilibrium associated to a deformation profile $\mathbf{r}_0$ and an energy profile $\mathbf{e}_0$ such that $\mathbf{e}_0 > \mathbf{r}_0^2/2$ (see \eqref{gibbs}). We denote by $\beta_0$ and $\lambda_0$ the potential profiles associated to $\mathbf{r}_0$ and $\mathbf{e}_0$: $$(\beta_0,\lambda_0):=\Psi(\mathbf{e}_0,\mathbf{r}_0)\ .$$
We assume that \begin{equation} H\left(\mu_0^N \vert \mu^N_{\beta_0(\cdot),\lambda_0(\cdot)}\right)=o(N)\end{equation} and that the initial profiles are continuous.

 We also assume that the energy moments are bounded: let us suppose that there exists a positive constant $C$ which does not depend on $N$ and $t$, such that \begin{equation} \forall \ k \geqslant 1,\ \mu_t^N \left[\sum_{x \in \T_N} e_x^{k} \right] \leqslant  (C k)^k \ N\ . \label{mom} \end{equation}
 Let $G$ be a continuous function on the torus $\T$ and $\varphi$ be a local function which satisfies the following property: there exists a finite subset $\Lambda \subset \Z$ and a constant $C>0$ such that, for all $\omega \in \Omega^N$, $ \varphi(\omega) \leqslant C\left(1+\sum_{i \in \Lambda} e_i(\omega)\right)$. Then,  \begin{equation} \mu_t^N \left[\left\vert \frac{1}{N} \sum_x G\left(\frac{x}{N}\right) \tau_x \varphi - \int_{\T} G(y) \ {\tilde{\varphi}(\mathbf{e}(t,q),\mathbf{r}(t,q))  dq} \right\vert \right] \xrightarrow[N \to \infty]{} 0 \end{equation} {where $\tilde{\varphi}$ is the grand-canonical expectation of $\varphi$: in other words, for any $(\mathbf{e}, \mathbf{r})\in \R^2$ such that $\mathbf{e} > \mathbf{r}^2/2$, let $(\beta,\lambda)=\Psi(\mathbf{e},\mathbf{r})$ then} \begin{equation} \tilde{\varphi}(\mathbf{e},\mathbf{r})=\mu_{\beta,\lambda}[\varphi]=\int_{(\R\times\R)^\Z} \varphi(\omega)\ d\mu_{\beta,\lambda}(\omega)\ . \end{equation} Besides,       $\mathbf{e}$ and $\mathbf{r}$ are defined on $\R_+ \times \T$  and  are solutions of \begin{equation} \left\{\begin{aligned} \d_t \mathbf{r} &=\frac{1}{\gamma} \d^2_q \mathbf{r}\ , \\
 \d_t \mathbf{e} & = \frac{1}{2\gamma} \d^2_q\left(\mathbf{e}+\frac{\mathbf{r}^2}{2}\right)\ ,\end{aligned} \right. \quad  q \in \T, \ t \in \R, \label{hydro} \end{equation} with the initial conditions $\mathbf{r}(\cdot,0)=\mathbf{r}_0(\cdot)$ and $\mathbf{e}(\cdot,0)=\mathbf{e}_0(\cdot)$.
 \end{theo}

\textbf{Remarks.} \begin{enumerate} \item In order to prove the theorem, we shall show afterwards that $H\left(\mu_t^N\vert \nu_{\chi_t(\cdot)}^N \right)=o(N). $

Here $\nu_{\chi_t(\cdot)}^N$ is a probability measure which is close to the  Gibbs local equilibrium  $\mu^N_{\beta(t,\cdot),\lambda(t,\cdot)}$ \eqref{gibbs}. The functions $(\beta(t,\cdot),\lambda(t,\cdot))$ are still related to $\mathbf{e}(t,\cdot)$ and $\mathbf{r}(t,\cdot)$ by \eqref{rel2}. 

This fact allows to establish the hydrodynamic limit in the sense given in the theorem. For a proof, we refer the reader to \cite{MR1231642}, {Corollary 2.2},  \cite{MR1707314}   or  \cite{BOmanus}.

\item Let us notice that the functions $\mathbf{e}, \mathbf{r}, \beta$ and $\lambda$ are smooth when $t >0$, since the system of partial differential equations is parabolic. Moreover, the function ${\beta}^{-1}= \mathbf{e}-\mathbf{r}^2/2$ satisfies \begin{equation} \d_t \left(\frac{1}{\beta}\right) = \frac{1}{2\gamma} \ \d_q^2\left(\frac{1}{\beta}\right)+ \frac{1}{\gamma} \ \vert \d_q r\vert^2 \geqslant \frac{1}{2\gamma} \d_q^2\left(\frac{1}{\beta}\right)\ . \end{equation} The supersolutions of the heat equation follow the minimum principle. Consequently, since there exists $c >0$ such that the initial profile $\beta_0$ has the following property $$\forall \ q \in \T_N,\ \beta_0(q) \geqslant c >0\ , $$ then we know that the function $\beta$ satisfies: \begin{equation}\forall \ q \in \T_N,\ \forall \ t \in [0,T], \  {\beta_t(q)} \geqslant c >0\ .\label{eq:beta}\end{equation}
 
\item After some integrations by parts, a simple computation shows that \begin{equation} \d_t\left\{ \int_{\T} S(\mathbf{r}(t,q),\mathbf{e}(t,q)) \ dq \right\} = \frac{1}{2\gamma} \int_{\T} \left[\frac{\d_q\beta(t,q)}{\beta(t,q)}\right]^2 + 2 \beta(t,q) \ [\d_qr(t,q)]^2 \ dq \geqslant 0 \end{equation} when $\mathbf{r}$ { and} $\mathbf{e}$ are the solutions of the hydrodynamic equations \eqref{hydro}. This fact is in agreement with the second thermodynamic principle.
 
 \end{enumerate}
 
In Section \ref{sec:bounds}, we will show that the hypothesis on moments bounds \eqref{mom} holds for a wide class of initial local equilibrium states. Before stating the theorem, let us give some definitions.
 
We denote by $\mathfrak{S}_N(\R)$ the set of real symmetric matrices of size $N$. The { correlation} matrix $C \in \mathfrak{S}_{2N}(\R)$ of a { probability} measure $\nu$ on $\Omega^N$  is the symmetric matrix $C=(C_{i,j})_{1\leqslant i,j \leqslant 2N}$  defined by \begin{equation} C_{i,j}:= \begin{cases}  \nu[r_i r_j] &  i,j \in \{1,...,N\}\ , \\ \nu[r_i p_j] &    i \in \{1,...,N\}, \ j \in \{N+1,...,2N \}\ , \\
\nu[p_i r_j]  &  i \in \{N+1,...,2N \}, \ j \in \{1,...,N\}\ , \\ 
\nu[p_i p_j]  &  i,j \in \{N+1,...,2N\}\ . \end{cases} \label{cor} \end{equation}

{Let us denote by $\Sigma_N$ the subset of $\R^{2N}\times\mathfrak{S}_{2N}(\R)$ defined by the following condition:} \begin{equation} (m,C) \in \Sigma_N \Leftrightarrow \begin{cases} m_k=0 & \text{ for all } k=N+1\dots 2N\ , \\
C_{i,j}=0 & \text{ for all } i \neq j\ , \\ C_{i,i} >0 & \text{ for all } i=1\dots 2N\ , \\ C_{i,i}-m_i^2=C_{i+N,i+N} & \text{ for all } i=1\dots N\ . \end{cases}\end{equation}

{Precisely, it means that $m$ is of the form $m=(m_1,\dots,m_N,0,\dots,0)$, and $C$ is a diagonal matrix whose diagonal components can be written as $(m_1^2+\alpha_1,\dots,m_N^2+\alpha_N,\alpha_1,\dots,\alpha_N)$, where $\alpha_i >0$ for all $i=1\dots N$.}

{ For $(m,C) \in \Sigma_N$, we denote by $G_{m,C}(\cdot)$ the Gaussian measure with mean $m$ and correlations given by the matrix $C$. The covariance matrix of $G_{m,C}(\cdot)$ is thus $C-m^tm$.} 

\begin{lem}\label{lem:matrix}  Let $\lambda(\cdot)$ and $\beta(\cdot)$ be two functions of class $\mathcal{C}^1$ defined on $\T$, and $\mu_{\beta(\cdot),\lambda(\cdot)}^N$ be the Gibbs local equilibrium defined by \eqref{gibbs}. If we denote by $m_{\beta(\cdot),\lambda(\cdot)}$ and $C_{\beta(\cdot),\lambda(\cdot)}$  respectively the mean vector and the correlation matrix of $\mu_{\beta(\cdot),\lambda(\cdot)}^N$, then we have $$ (m_{\beta(\cdot),\lambda(\cdot)},C_{\beta(\cdot),\lambda(\cdot)})\in\Sigma_N \ \text{ and } \ \mu_{\beta(\cdot),\lambda(\cdot)}^N=G_{m_{\beta(\cdot),\lambda(\cdot)},C_{\beta(\cdot),\lambda(\cdot)}}.$$\end{lem} \begin{proof}  This result comes from the explicit formula of $\mu_{\beta(\cdot),\lambda(\cdot)}^N$ given in \eqref{gibbs}. 
First, notice that  each momentum $p_x$ is centered under $\mu_{\beta(\cdot),\lambda(\cdot)}^N$ and \begin{equation} \mu_{\beta(\cdot),\lambda(\cdot)}^N[r_x]=-\frac{\lambda}{\beta}\left(\frac{x}{N}\right)\ . \end{equation} Second, we easily obtain the following expressions:
\begin{equation} m_{\beta(\cdot),\lambda(\cdot)}=\left(-\frac{\lambda}{\beta}\left(\frac{0}{N}\right),\cdots,-\frac{\lambda}{\beta}\left(\frac{N-1}{N}\right), \underbrace{0, \cdots, 0}_N\right)\ , \end{equation}

\begin{equation} C_{\beta(\cdot),\lambda(\cdot)}=\begin{pmatrix} D & 0_N \\ 0_N & D' \end{pmatrix}\qquad \text{ where } \left\{ \begin{aligned} D & =\text{diag}\left(\cdots, \frac{1}{\beta(x/N)} + \frac{\lambda^2(x/N)}{\beta^2(x/N)}, \cdots \right)\ , \\
D' & =\text{diag}\left(\cdots, \frac{1}{\beta(x/N)}, \cdots \right). \end{aligned} \right. \end{equation}
\end{proof}


{ Now we state our second main theorem, which will be proved in Section \ref{sec:bounds}.}

 \begin{theo}\label{theo:moments} We assume that the initial probability measure $\mu_0^N$ is a convex combination of Gibbs local { equilibrium states}. More precisely, let $\sigma$ be a {probability} measure whose support is included in {$\Sigma_N$}. We assume that $\sigma$ satisfies: \begin{equation}\text{for all } k \geqslant 1,\  \int \left[K(m,C)\right]^k \ d\sigma(m,C) < \infty\ , \end{equation} {where $\ds K(m,C):= \sup_{i=1\dots N} \left\{ C_{i,i} \right\}$.} We define the initial probability measure $\mu_0^N$ by \begin{equation} \mu_0^N(\cdot)=\int G_{ m,C}(\cdot) \ d\sigma({ m,C})\ . \label{init} \end{equation} 
 Then, \eqref{mom} holds, and the conclusions of Theorem \ref{theo:hydro} are valid.
  
 \end{theo}
 
 \textbf{Remark.} As in \cite{BerKan}, we could consider a more general model, with a pinning potential. { Instead of the deformation $r_x$, we now introduce the position $q_x$ of the particle $x$.} The new pinning Hamiltonian is given by \begin{equation} \mathcal{H}_N^{p}=\sum_{x\in \T_N} \frac{p_x^2}{2} + \nu^2 \sum_{x\in\T_N} \frac{q_x^2}{2} + \sum_{ \substack{ \vert x-y \vert =1, \\ x,y \in \T_N}} \frac{(q_x-q_y)^2}{4}\ . \end{equation} The strength of the pinning potential is regulated by the parameter $\nu > 0$. The energy of site $x$ is now given by \begin{equation} e_x=\frac{p_x^2}{2} + \nu^2 \frac{q_x^2}{2} + \frac{1}{4} \sum_{y  ;  \vert x-y\vert =1} (q_x-q_y)^2\ . \end{equation}
The stochastic operator  $\S_N^{p}$ remains equal to $\S_N$, and the Liouville operator $\A_N^{p}$ can be written as follows: \begin{equation}\A_N^{p} = \sum_{x \in \T_N} \left\{ p_x \ \d_{q_x} - [(\nu^2-\Delta)q]_x \ \d_{p_x} \right\}\ , \end{equation} where $\Delta$ is the discrete Laplacian: $(\Delta u)_x=u_{x+1}+u_{x-1}-2u_x$.

Because of the presence of the pinning, the bulk dynamics conserves only one quantity: the total energy  $\sum_x e_x$. It follows that the Gibbs equilibrium measures $\mu_{\beta}^N$ are fully characterized by the temperature $\beta^{-1}$. Under $\mu_\beta^N$, the variables $p_x$ are independent of the $q_x$ and are independent identically Gaussian variables of variance $\beta^{-1}$. The $q_x$ are distributed according to a centered Gaussian process with covariances given by \begin{equation} \mu_{\beta}^N(q_xq_y)={ \Gamma(x-y),  \text{ such that } \left[(\nu^2-\Delta)\right]\Gamma(z)=\frac 1 \beta \mathds{1}_{z=0}\ .} \end{equation} Observe that there exists $C:=C(\nu)$ independent { of} $N$ such that $\left\vert \mu_{\beta}^N(q_xq_y) \right\vert \leqslant C^{-1} e^{-C \vert x-y\vert}$ for any $N \geqslant 1$.

These correlations make computations more technical, but the hydrodynamic limit can be established by following the proof here (in \cite{BerKan}, { Section 3.2}, a heuristic argument is given). Assume that the system is initially distributed according to a Gibbs local equilibrium associated to the energy profile $\mathbf{e}_0(\cdot)$, and define $\mathbf{e}(t,\cdot)$ as the evolved profile in the diffusive scale. Then, if the energy moments are bounded like \eqref{mom}, $\mathbf{e}$ is the solution of the following heat equation \begin{equation} \left\{ \begin{aligned} & \d_t\mathbf{e}  = \d_q(D(\mathbf{e}) \d_q \mathbf{e})\ , \\ & \mathbf{e}(0,\cdot)=\mathbf{e}_0(\cdot)\ , \end{aligned} \right. \quad  q \in \T, \ t \in \R, \end{equation} where $D(\mathbf{e})$ is the \textit{diffusivity} given by \begin{equation} D:=D(\mathbf{e})=\frac{1/\gamma}{2+\nu^2+\sqrt{\nu^2(\nu^2+4)}}\ . \end{equation}
In our model, where the state space is not compact, what matters is the existence of moments bounds. We will see { in Section \ref{sec:bounds}} that this existence can be easily justified by following the same ideas which work for the unpinned model.

For the sake of simplicity, we will denote by $\mathbf{e}_t(\cdot)$, $\mathbf{r}_t(\cdot)$, $\lambda_t(\cdot)$ and $\beta_t(\cdot)$, respectively, the functions $q \to \mathbf{e}(t,q)$, $q \to \mathbf{r}(t,q)$, $q \to \lambda(t,q)$, and $q\to \beta(t,q)$ defined on $\T$.

 \subsection{Ergodicity of the Infinite Velocity-flip Model}\label{subsec:ergo}
 
 We conclude this part by giving the theorem of ergodicity, which is proved in \cite{BOmanus}, { Sections 2.2 and 2.4.2}, by following the ideas of \cite{MR1278883}. Let us define, for all finite { subsets} $\Lambda \subset \Z$, and for two probability measures $\nu$ and $\mu$ on $\Omega:=(\R\times\R)^{\Z}$, the restricted relative entropy \begin{equation} H_{\Lambda}(\nu\vert\mu):=H(\nu_{\Lambda}\vert\mu_{\Lambda}) \end{equation} where $\nu_{\Lambda}$ and $ \mu_{\Lambda}$ are the marginal distributions of $ \nu$  and $ \mu $  on $ \Omega.$

The Gibbs states in infinite volume are the probability measures $\mu_{\beta,\lambda}$ on $\Omega$ given by \begin{equation} d\mu_{\beta,\lambda}:=\prod_{x \in \Z} \frac{e^{-\beta e_x -\lambda r_x}}{Z(\beta,\lambda)}dr_x dp_x\ . \end{equation} {The formal generator of the infinite dynamics is denoted by $\L$.}
 
 \begin{theo} \label{theo:ergod} Let $\nu$ be a probability measure on the configuration space $\Omega$ such that \begin{enumerate}
 \item $\nu$ has finite density entropy: there exists $ C >0,  \text{ such that for all finite {subsets} } \Lambda \subset \Z$, \begin{equation} H_{\Lambda}(\nu \vert \mu_{\ast}) \leqslant C \vert \Lambda \vert\ ,\end{equation} with $\mu_{\ast}:=\mu_{1,0}$ a reference Gibbs measure on $(\R \times \R)^{\Z}$,
 \item $\nu$ is translation invariant,
 \item $\nu$ is stationary, i.e. for any compactly supported and differentiable function $F(\r,\p)$, \begin{equation} \int \A (F)\ d\nu=0\ . \end{equation}
 \item { the conditional probability distribution of $\p$ given the probability distribution of $\r$, denoted by $\nu(\p \vert \r)$,} is invariant by any flip ${ \p \to \p^x}$, with $x \in \Z$.
 \end{enumerate}
 
 Then, $\nu$ is a mixture of { infinite} Gibbs states.
 \end{theo} 
 
\begin{cor} \label{cor:ergod} If $\nu$ is a probability measure on $\Omega$ satisfying 1, 2 and if $\nu$ is stationary in the sense that: for any compactly supported and differentiable function $F(\r,\p)$, \begin{equation} \int \L(F) \ d\nu=0\ , \end{equation} then $\nu$ is a mixture of { infinite} Gibbs states.

\end{cor}

The outline of the rest of the paper is as follows. In the next section we expose the strategy of the proof. We introduce the relative entropy $H_N(t)$ of $\mu_t^N$ with respect to a corrected local equilibrium, and  we prove a Gronwall estimate of the entropy production of the form \begin{equation} \d_t H_N(t) \leqslant C \ H_N(t)+o(N)\ , \label{gron} \end{equation} where $C >0$ does not depend on $N$. In Section \ref{sec:bounds} we prove Theorem \ref{theo:moments}.
 
We suppose that $t$ belongs to a compact set $[0,T]$, $T$ fixed. All estimates are uniform in $t \in [0,T]$.

\section{Entropy Production } \label{sec:entropy}

\subsection{Introduction to the Method}\label{subsec:method}

For the sake of simplificity, we denote all couples of the form $\ds (\beta(\cdot),\lambda(\cdot))$ by $\chi(\cdot)$. 

The corrected Gibbs local equilibrium state $\nu^N_{\chi_t(\cdot)}$ is defined by \begin{equation}\nu^N_{\chi_t(\cdot)}:=\frac{1}{Z(\chi_t(\cdot))}  \prod_{x \in \T_N} \exp\left(-\beta_t\left(\frac{x}{N}\right) e_x-\lambda_t\left(\frac{x}{N}\right)  r_x+\frac{1}{N}F\left(t,\frac{x}{N}\right)\cdot \tau_xh(\r,\p) \right)dr_x dp_x\end{equation} where  $Z(\chi_t(\cdot))$ is the partition function and $F$, $h$ are functions which will be precised later on. { The notation $a \cdot b $ still stands for the usual scalar product between $a$ and $b$}. An estimate of the partition function $Z(\chi_t(\cdot))$ is performed in  Appendix \ref{appa}. 

We are going to use the relative entropy method, with the corrected local Gibbs state $\nu_{\chi_t(\cdot)}^N$ instead of the usual one  $\mu^N_{\chi_t(\cdot)}$. We define \begin{equation} H_N(t):=H\left(\mu_t^N \vert \nu^N_{\chi_t(\cdot)}\right)=\int_{\Omega^N} f_t^N(\omega) \log \frac{f_t^N(\omega)}{\phi_t^N(\omega)} d\mu^N_{1,0}(\omega)\ , \end{equation} where  $f_t^N$ { is the density of $\mu_t^N$ with respect to the reference measure $\mu_{1,0}^N$. This is a solution, in the sense of the distributions, of the Fokker-Planck equation}
\begin{equation} \d_t f_t=N^2\ \L_N^*f_t \end{equation} {where $\L_N^*=-\A_N+\gamma S_N$ is the adjoint of $\L_N$ in $\mathbb{L}^2(\mu_{1,0}^N)$.  In the same way, $\phi_t^N$ is the density of $\nu_{\chi_t(\cdot)}^N$ with respect to $\mu_{1,0}^N$ (which here is easily computable).}

Thus, our purpose is now to prove \eqref{gron}. We begin with the following lemma.

\begin{lem}\label{entropy}
\begin{equation} \d_t H_N(t) \leqslant \int \frac{1}{\phi_t^N}\left(N^2\L_N^*\phi_t^N-\d_t\phi_t^N\right) f_t^N d\mu_{1,0}=\mu_t^N \left[  \frac{1}{\phi_t^N}\left(N^2\L_N^*\phi_t^N-\d_t\phi_t^N\right) \right]\ . \end{equation}

\end{lem}

\begin{proof} The case where $f_t^N$ is { smooth} is proved in  \cite{MR1707314}, Chapter 6, Lemma 1.4.  Here, we do not know that $f_t^N$ is smooth, so that we refer the reader to the proof in \cite{BerOlla}, { Section 3.2}, which can be easily followed.  \end{proof}

Now, we choose the correction term. We consider  \begin{equation} \left\{ \begin{aligned} F\left(t,\frac{x}{N}\right) & :=\left(-\beta'_t\left(\frac{x}{N}\right),  - \lambda'_t\left(\frac{x}{N}\right) \right)\ , \\
\tau_xh(\r,\p) & := \left(\frac{r_x}{2\gamma}\left(p_{x+1}+p_x+\frac{\gamma}{2}r_x\right), \frac{p_{x+1}}{\gamma} \right)\ . \end{aligned}\right.  \label{fonc} \end{equation} Thus, \begin{equation} \phi_t^N(\r,\p)=\frac{(Z(1,0))^n}{Z(\chi_t(\cdot))} \prod_{x\in \T_N}  \exp\left(e_x\left(-\beta_t\left(\frac{x}{N}\right)+1\right)-\lambda_t\left(\frac{x}{N}\right)r_x+\frac{1}{N} F\left(t,\frac{x}{N}\right) \cdot \tau_xh(\r,\p)\right)\ . \end{equation}

We define $\xi_x:=(e_x,r_x)$ and $\eta(t,{ q}):=(\mathbf{e}(t,{ q}),\mathbf{r}(t,{ q})).$ {If $f$ is a vectorial function, we denote  its differential by $Df$.}

In Appendix \ref{appa}, the following technical result is proved.

\begin{prop} \label{prop} The term $ (\phi_t^N)^{-1}\left(N^2\L_N^*\phi_t^N-\d_t\phi_t^N\right)$ is given by the sum of five terms in which a microscopic expansion up to the first order appears. In other words,  \begin{multline} \frac{1}{\phi_t^N}\left(N^2\L_N^*\phi_t^N-\d_t\phi_t^N\right) = \\
=\sum_{k=1}^5 \sum_{x \in \T_N} v_k\left(t,\frac{x}{N}\right) \left[J_x^k-H_k\left({\eta}\left(t,\frac{x}{N}\right)\right)-(DH_k)\left({\eta}\left(t,\frac{x}{N}\right)\right)\cdot \left({\xi}_x-{\eta}\left(t,\frac{x}{N}\right)\right)\right] +o(N)\label{tay} \end{multline} where $$\begin{array}{| c | c | c | c |}
\hline k & J_x^k & H_k(\mathbf{e},\mathbf{r}) & v_k(t,{ q}) \\
\hline \hline \ds 1 & \ds p_x^2+r_xr_{x-1}+2\gamma p_x r_{x-1} & \ds \mathbf{e}+\frac{\mathbf{r}^2}{2} & \ds \frac{-1}{2\gamma} \d^2_q \beta(t,{ q}) \\ 
2 & r_x+\gamma p_x & \mathbf{r} & \ds -\frac{1}{\gamma} \d^2_q\lambda(t,{ q}) \\
3 & p_x^2\ (r_x+r_{x-1})^2 & \ds (2\mathbf{e}-\mathbf{r}^2)  \left(\mathbf{e}+\frac{3}{2}\mathbf{r}^2\right) &\ds  \frac{1}{4\gamma} [\d_q \beta(t,{ q})]^2 \\
4 & p_x^2 \ (r_x+r_{x-1}) & \mathbf{r} \ (2\mathbf{e}-\mathbf{r}^2)& \ds \frac{1}{\gamma} \d_q\beta(t,{ q}) \ \d_q\lambda(t,{ q}) \\
5 & p_x^2 & \ds \mathbf{e}-\frac{\mathbf{r}^2}{2} & \ds \frac{1}{\gamma} [\d_q \lambda(t,{ q})]^2 \\ \hline \end{array}$$
\end{prop}

A priori  the first term on the right-hand side of \eqref{tay} is of order $N$, but we want to take advantage of these microscopic Taylor expansions. First, we need to {cut-off large energies in order to work with bounded variables only}. Second, the strategy consists in performing a one-block estimate: we replace the empirical truncated current, which is averaged over a microscopic box centered at $x$,  by its mean with respect to a Gibbs measure with the parameters corresponding to the microscopic averaged profiles.

A one-block estimate will be performed for each term of the form  \begin{equation} 
\sum_{x \in \T_N} v_k\left(t,\frac{x}{N}\right) \left[J_x^k-H_k\left({\eta}\left(t,\frac{x}{N}\right)\right)-(DH_k)\left({\eta}\left(t,\frac{x}{N}\right)\right)\cdot\left({\xi}_x-{\eta}\left(t,\frac{x}{N}\right)\right)\right]\ . \label{taylor}
\end{equation}

In the following the index $k$ is omitted, whenever this does not cause confusion. We follow the lines of the proof given in \cite{BOmanus}, { Section 3.3} and inspired from \cite{MR1231642}.  A sketch of the proof for the one-block estimate is given in Appendix \ref{appb}.

\subsection{Cut-off of Large Energies}\label{subsec:cutoff}

For $x \in \T_N$, we define $A_{x,M}:=\{e_x+e_{x-1} \leqslant M\}$,  $J_{x,M}:=J_x \ \mathds{1}_{A_{x,M}}, $ and ${\xi}_{x,M}:={\xi}_{x}\ \mathds{1}_{e_x \leqslant M}. $ 

Then, $J_{x,M}$ and $\xi_{x,M}$ are bounded by $C(M)>0$.

We use twice the Cauchy-Schwartz inequality to write \begin{align} \mu_t^N\left[ \sum_{x\in\T_N} v\left(t,\frac{x}{N}\right) \ J_x \ \mathds{1}_{A_{x,M}^c}\right]  & \leqslant \mu_t^N \left[ \left(\sum_{x \in \T_N} v^2\left(t,\frac{x}{N}\right) J_x^2 \right)^{1/2} \ \left(\sum_{x \in \T_N} \mathds{1}_{A_{x,M}^c} \right)^{1/2} \right] \notag \\ 
 & \leqslant \left(\mu_t^N \left[ \sum_{x \in \T_N} v^2\left(t,\frac{x}{N}\right) J_x^2 \right]\right)^{1/2} \ \left(\mu_t^N \left[\sum_{x \in \T_N} \mathds{1}_{A_{x,M}^c} \right]\right)^{1/2}. \end{align}
First, $v^2\left(t,x/N\right)$ is bounded by a constant which does not depend on $N$. Second, the term $J_x^2$ can be bounded above by the squared energy $e_x^2$. The hypothesis \eqref{mom} shows that  there exists $C_0$ which does not depend on $N$ such that \begin{equation} \left(\mu_t^N \left[ \sum_{x \in \T_N} v^2\left(t,\frac{x}{N}\right) J_x^2 \right]\right)^{1/2} \leqslant C_0 \ N^{1/2}\ . \end{equation}
Moreover, Markov inequality proves that \begin{equation} \mu_t^N \left[\sum_{x \in \T_N} \mathds{1}_{A_{x,M}^c} \right] \leqslant \sum_{x \in \T_N} \mu_t^N \left[ \mathds{1}_{e_x > M/2} \right] + \mu_t^N \left[ \mathds{1}_{e_{x-1} > M/2} \right] \leqslant \frac{4}{M} \sum_{x \in \T_N} \mu_t^N[e_x] \leqslant \frac{C_1}{M} \ N\ .\end{equation}
Finally, we obtain a constant $C$ independent { of} $N$ such that \begin{equation} \mu_t^N\left[ \sum_{x\in\T_N} v\left(t,\frac{x}{N}\right) \ J_x \ \mathds{1}_{A_{x,M}^c}\right] \leqslant CN \ \varepsilon(M)\ . \end{equation}
Observe that this estimate is in agreement with the Gronwall inequality we want to prove, since we are going to divide by $N$. Thus, the error term is of order $1/M$ that goes to 0 as $M \to \infty$.

Consequently, $J_x$ can be replaced by $J_{x,M}$ in \eqref{taylor}, and similarly, ${\xi}_x$ can be replaced by ${\xi}_{x,M}$.

\subsection{One-block Estimate}\label{subsec:estimate}

Now we prove that \begin{multline} \frac{1}{N} \mu_t^N \left[\sum_{x \in \T_N} v\left(t,\frac{x}{N}\right) \left[J_{x,M}-H\left({\eta}\left(t,\frac{x}{N}\right)\right)-(DH)\left({\eta}\left(t,\frac{x}{N}\right)\right)\cdot\left({\xi}_{x,M}-{\eta}\left(t,\frac{x}{N}\right)\right)\right]\right] \leqslant \\ \leqslant C \ \frac{H_N(t)}{N} + \varepsilon(N,M) \end{multline} with $\varepsilon(N,M) \xrightarrow[M \to \infty , N\to \infty]{} 0. $

We denote by $\Lambda_{\ell}(y)$ the box of length $\ell$ centered around $y$. We introduce the microscopic average profiles \begin{equation}{\eta}_{\ell,M}(y):=\frac{1}{\ell} \sum_{j \in \Lambda_{\ell}(y)} {\xi}_{j,M}\ .\end{equation}
We split $\T_N$ into $p=N/\ell$ boxes $\Lambda_{\ell}(x_j)$ centered at $x_j$.  Here $\ell$ is assumed to divide $N$ for simplicity. We will first let $N \to \infty$, then $\ell \to \infty$ and then $M \to \infty$.

First of all, we want to replace \begin{equation}\frac{1}{N} \sum_{x \in \T_N} v\left(t,\frac{x}{N}\right) J_{x,M}\end{equation} by \begin{equation}{\frac{1}{p}} \sum_{j=1}^p v\left(t,\frac{x_k}{N}\right)\left[ \frac{1}{\ell} \sum_{i \in \Lambda_{\ell}(x_j)} J_{i,M}\right]\ .\end{equation}

The error term produced during this step can be written as \begin{equation}\vert R_N \vert = \frac{1}{N} \left\vert \sum_{j=1}^p \sum_{i \in \Lambda_{\ell}(x_j)} \left[v\left(t,\frac{i}{N}\right)-v\left(t,\frac{x_j}{N}\right)\right] J_{i,M} \right\vert  \leqslant C_1(M) \frac{\ell}{N} \ .\end{equation}
The last inequality comes from the smoothness of $v$, more precisely \begin{equation} \left\vert v\left(t,\frac{i}{N}\right)-v\left(t,\frac{x_j}{N}\right)\right\vert \leqslant C_0 \frac{\ell}{N}\ .\end{equation}

Similarly, we perform the same estimates for the other terms and it remains to prove that \begin{multline} \mu_t^N \left[\frac{1}{p}\sum_{j=1}^p v\left(t,\frac{x_j}{N}\right) \left\{\frac{1}{\ell} \sum_{i \in \Lambda_{\ell}(x_j)} J_{i,M}-H\left({\eta}\left(t,\frac{x_j}{N}\right)\right)\right.\right.\\ \left.\left. -(DH)\left({\eta}\left(t,\frac{x_j}{N}\right)\right)\cdot\left({\eta}_{\ell,M}(x_j)-{\eta}\left(t,\frac{x_j}{N}\right)\right)\right.\Bigg\}\right]  \label{limit01} \end{multline} vanishes as $M, N, \ell \to \infty$, the limit in $N$ taken first, then the limit in $\ell$ and finally the limit in $M$. The {additive term which appears after performing this replacement} can be bounded above by a term $\varepsilon_{N,M,\ell}$ which depends on $N,M$ and $\ell$, but which is independent { of} the particular splitting of $\T_N$ into $p$ boxes. This term is of order $o(N)$ in the Gronwall inequality we want to prove, in the sense that \begin{equation}\lim_{M \to \infty} \lim_{\ell \to \infty} \lim_{N \to \infty} {N}^{-1} \ \mu_t^N[\varepsilon_{N,M,\ell}]=0\ .\end{equation}
Now we want to perform a one-block estimate. The main idea consists in replacing $ {\ell}^{-1} \sum_{i \in \Lambda_{\ell}(x_j)} J_{i,M} $ by $H({\eta}_{\ell,M}(x_j))$. 
This is achieved thanks to the ergodicity of the dynamics (see Theorem \ref{theo:ergod}). In order to use this ergodicity property, we  have to work with a space translation invariant measure. To obtain such a probability measure, we introduce a second average over the $x_j,\ 1\leqslant j \leqslant p$. { For each $k \in \{0,..., \ell-1\}$, we can split $\T_N$ into $p$ disjoint boxes of length $\ell$ by writing}\begin{equation}\forall \ k=0,...,\ell-1,\ \T_N=\bigcup_{j=1}^p \Lambda_{\ell}(x_j+k)\ .\end{equation}Then, we average the different splittings mentioned above. More precisely, in Appendix \ref{appb} we recall how to prove \begin{equation} \limsup_{M \to \infty} \limsup_{\ell \to \infty} \limsup_{N \to \infty} \frac{1}{\ell} \sum_{k=0}^{\ell-1} \mu_t^N\left[\frac{1}{p \ell}\sum_{j=1}^p \left\vert  v\left(t,\frac{[x_j + k]}{N}\right)\sum_{i \in \Lambda_{\ell}(x_j+k)} J_{i,M}-H({\eta}_{\ell,M}(x_j+k))\right\vert \right]=0. \label{oneb} \end{equation}

\subsection{Large Deviations }\label{subsec:deviations}

The previous estimates are valid for any splitting of $\T_N$ into $p$ boxes of length $\ell$. Thus, it would be sufficient to prove \eqref{limit01} with every $x_i$ replaced by $x_i+k$ for arbitrary $k \in \{1,...,\ell-1\}$. Consequently, it is sufficient to prove \eqref{limit01} in an averaged form. Then, from the one-block estimate, we have to deal with \begin{equation} \frac{1}{\ell} \sum_{k=0}^{\ell-1} \mu_t^N\left[\frac{1}{N}\sum_{j=1}^p v\left(t,\frac{[x_j+k]}{N}\right) \Omega\left({\eta}_{\ell,M}(x_j+k),{\eta}\left(t,\frac{[x_j+k]}{N}\right)\right)\right], \label{limit02} \end{equation} where $\Omega(\mathbf{w},\mathbf{u}):=H(\mathbf{w})-H(\mathbf{u})-DH(\mathbf{u})\cdot(\mathbf{w}-\mathbf{u}).$

By definition of the entropy, for any $\alpha >0$ and any positive measurable function $f$ we have \begin{equation} \int f \ d\mu \leqslant \frac{1}{\alpha} \left\{ \log \left(\int e^{\alpha f} \ d\nu \right) + H(\mu \vert \nu) \right\}\ . \label{entropin}\end{equation}
This inequality, known as the \textit{entropy inequality}, allows to show that: for { any $\alpha >0$}, \eqref{limit02} is less than or equal to \begin{equation}\frac{H_N(t)}{\alpha\ N} + \frac{1}{\ell} \sum_{k=0}^{\ell-1}  \frac{1}{\alpha N}  \log \nu_{\chi_t(\cdot)}^N\left[e^{\alpha \ell \sum_{j=1}^p v\left(t,\frac{[x_j+k]}{N}\right)  \Omega\left({\eta}_{\ell,M}(x_j+k),{\eta}\left(t,\frac{[x_j+k]}{N}\right)\right)} \right]\ .\end{equation} Notice that the last integral converges because all quantities are bounded.

The first term is in agreement with the Gronwall inequality we want to obtain. We look at the second term. Since we have arranged the sum over $p$ disjoint blocks which are independently distributed by $\nu_{\chi_t(\cdot)}^N$, the second term is equal to 
\begin{equation} \frac{1}{\ell} \sum_{k=0}^{\ell-1}  \frac{1}{\alpha N} \sum_{j=1}^p  \log \nu_{\chi_t(\cdot)}^N\left[e^{\alpha \ell v\left(t,\frac{[x_j+k]}{N}\right)  \Omega\left({\eta}_{\ell,M}(x_j+k),{\eta}\left(t,\frac{[x_j+k]}{N}\right)\right)} \right]. \end{equation}
We are going to show that this expression vanishes as $M, N, \ell \to \infty$ by using the large deviation properties of the measure $\nu_{\chi_t(\cdot)}^N$, that locally is almost homogeneous. In fact, by using the smoothness for the various involved functions, we can substitute the inhomogeneous product measure $\nu_{\chi_t(\cdot)}^N$ restricted to $\Lambda_{\ell}(x_j+k)$ with the homogeneous product  measure $\mu_{\chi_t([x_j+k]/N)}^N$, in each expectation of the previous expression. More precisely, we have the following lemma.

\begin{lem}\label{lem:replace}
\begin{equation} M_1(N,\ell,k,M):=\frac{1}{\alpha N} \sum_{j=1}^p \log \nu_{\chi_t(\cdot)}^N\left[e^{\alpha \ell \left\vert v\left(t,\frac{[x_j+k]}{N}\right)  \Omega\left({\eta}_{\ell,M}(x_j+k),{\eta}\left(t,\frac{[x_j+k]}{N}\right)\right)\right\vert} \right]\end{equation} can be replaced by \begin{equation} M_2(N,\ell,k,M):=\frac{1}{\alpha N} \sum_{j=1}^p \log \mu^N_{\chi_t([x_j+k]/N)} \left[e^{\alpha \ell \left\vert v\left(t,\frac{[x_j+k]}{N}\right)  \Omega\left({\eta}_{\ell,M}(x_j+k),{\eta}\left(t,\frac{[x_j+k]}{N}\right)\right)\right\vert} \right].\end{equation}  The {difference between these two terms is less than or equal to a small term which} depends on $\ell$ (but not on $k$) and vanishes in the $N$ limit: there exists a constant $C(\ell,M,N)$ which does not depend on $k$ such that \begin{equation} M_1(N,\ell,k,M)-M_2(N,\ell,k,M)\leqslant C(\ell,M,N) \ \text{ and } \ C(\ell,M,N) \xrightarrow[N \to \infty]{} 0\ .\end{equation}\end{lem}

{\textbf{Remark.} In the following, we will prove that }\begin{equation}\limsup_{M \to \infty} \limsup_{\ell \to \infty} \limsup_{N \to \infty} M_2(N,\ell,k,M)=0\ .\end{equation} {In addition to this lemma, this implies that} \begin{equation}\limsup_{M \to \infty} \limsup_{\ell \to \infty} \limsup_{N \to \infty} M_1(N,\ell,k,M)=0\ ,\end{equation} {since $M_1(N,\ell,k,M)$ is always nonnegative, and we know that, for all sequences $a_n$ and $b_n$,} \begin{equation}\limsup a_n \leqslant \limsup (a_n-b_n) + \limsup b_n \ .\end{equation}

\begin{proof} For each $j \in \{1,...,p\}$, the function \begin{equation}F_j:= \exp\left\{ \alpha \ell \left\vert v\left(t,\frac{[x_j+k]}{N}\right) \ \Omega\left({\eta}_{\ell,M}(x_j+k),{\eta}\left(t,\frac{[x_j+k]}{N}\right)\right)\right\vert \right\}\end{equation} is bounded above by $e^{C \ell}, \ C>0$ (since $\eta_{\ell,M}$ is bounded and $t$ belongs to a compact set), and depends on the configuration only through the coordinates in $\Lambda_\ell(x_j+k)$. Thus, each expectation appearing in the sum can be taken w.r.t  the restriction to $\Lambda_\ell(x_j+k)$ of $\nu_{\chi_t(\cdot)}^N$. These restrictions are inhomogeneous product measures but with { slowly} varying parameters and hence, each term $\log \nu_{\chi_t(\cdot)}^N[F_j]$ can be replaced by $\log \mu^N_{\chi_t([x_j+k]/N)} [F_j]$ with a small { error}.

Indeed, the difference between these two terms is equal to \begin{equation}\log \mu^N_{\chi_t([x_j+k]/N)} \left[ 1+ \frac{F_j \ (h_j-1)}{\mu^N_{\chi_t([x_j+k]/N)} [F_j]} \right]=\log\left(1+ \frac{\mu^N_{\chi_t([x_j+k]/N)}[F_j \ (h_j-1)]}{\mu^N_{\chi_t([x_j+k]/N)} [F_j]}\right)\end{equation}
with \begin{multline} h_j:= \exp \Bigg\{\sum_{i \in \Lambda_\ell(x_j+k)} {\xi}_{i,M} \cdot \left[ \chi_t\left(\frac{i}{N}\right) - \chi_t\left(\frac{x_j+k}{N}\right)\right] + \\ { - } \frac{1}{N} F\left(t,\frac{i}{N}\right) \cdot \tau_ih \ {  + }\left[\log Z\left(\chi_t\left(\frac{i}{N}\right)\right)-\log Z\left(\chi_t\left(\frac{x_j+k}{N}\right)\right)\right]\Bigg\}\ .\end{multline}
The inequality {$\log(1+x) \leqslant \vert x \vert$ (true for any real $x$)} and the fact that $\mu^N_{\chi_t([x_j+k]/N)} [F_j] \geqslant 1$ reduces us to estimate \begin{equation}\mu^N_{\chi_t([x_j+k]/N)}[\vert F_j (h_j-1)\vert]\ .\end{equation}
By using the smoothness of $\chi_t$ and the inequality $\vert e^x-1 \vert \leqslant \vert x \vert e^{\vert x \vert}$, one easily shows that there exist positive constants $C_0$, $C(\ell)$, and $\bar{\beta}$ which do not depend on $j$ such that \begin{equation*}\vert F_j (h_j-1)\vert \leqslant \frac{C( \ell) \ell}{N} \left(\sum_{ i \in \Lambda_\ell(x_j+k)} \left[e_{i,M} + e_{i+1,M} +1 \right] \right) \exp \left\{ \frac{C_0 \ell}{N} \sum_{ i \in \Lambda_\ell(x_j+k)} \left[e_{i,M} + e_{i+1,M} +1 \right] \right\}\ ,\end{equation*}  \begin{equation} \text{and } \ \ \left. \frac{d\mu^N_{\chi_t([x_j+k]/N)}}{d\mu_{\bar{\beta},0}^N} \right\vert_{\Lambda_\ell(x_j+k)} \leqslant C(\ell)\ . \end{equation}
Hence, the total error performing by these replacements is bounded above: $$ M_1(N,\ell,k,M)-M_2(N,\ell,k,M) \leqslant \frac{1}{\alpha N} C_1(\ell,M) \ \mu^N_{\bar{\beta},0}\left[\exp \left\{ \frac{C_0}{p} \sum_{i \in \Lambda_\ell(0)} [e_{i,M} + e_{i+1,M} +1] \right\}\right]$$ for some positive constant $C_1(\ell,M)$.

It trivially goes to 0 as $N$ goes to infinity for each given fixed $\ell$. \end{proof}

Lastly, we have to show that the limit \begin{equation}\limsup_{M \to \infty} \limsup_{\ell \to \infty} \limsup_{ N \to \infty} \frac{1}{\ell} \sum_{k=0}^{\ell-1}  \frac{1}{\alpha N} \sum_{j=1}^p  \log \mu^N_{\chi_t([x_j+k]/N)}\left[e^{\alpha \ell v\left(t,\frac{[x_j+k]}{N}\right) \ \Omega\left({\eta}_{\ell,M},{\eta}\left(t,\frac{[x_j+k]}{N}\right)\right)} \right]\end{equation}
vanishes. Here, $ {\eta}_{\ell,M}:=\eta_{\ell,M}(0)={\ell}^{-1} \sum_{i \in \Lambda_{\ell}(0)} { \xi_{i,M}}$\ .

The limit in $p$ results in an integral over $\T$ because we have a Riemann sum. Moreover, the integral does not depend on $k$ so that the averaging over $k$ disappears in the $p$ limit. Hence, the point is to estimate 
\begin{equation} \limsup_{M \to \infty} \limsup_{\ell \to \infty} \frac{1}{\alpha \ell} \int_{\T} { \log } \ \mu^N_{\chi_t({ q})} \left[e^{ \alpha \ell v\left(t,{ q}\right) \ \Omega\left({\eta}_{\ell,M},{\eta}\left(t,{ q}\right)\right)}\right] \ d{ q}\ .\end{equation}
According to Laplace-Varadhan theorem applied to these product measures $\mu_{\chi_t({ q})}^N$, and { according to the dominated convergence theorem}, the previous limit is equal to \begin{equation} \limsup_{M \to \infty} \frac{1}{\alpha} \int_{\T} \sup_{\mathbf{z} \in \R^2} \left\{\alpha v(t,{ q}) \ \Omega(\mathbf{z},{\eta}(t,{ q})) - I_M(\mathbf{z}, {\eta}(t,{ q}))\right\} d{ q}\ , \label{limsup} \end{equation}
where $I_M(\mathbf{z},{\eta}(t,{ q}))$ is the rate function of the sequence $\left\{ k^{-1} \sum_{i=1}^k \xi_{i,M}\right\}_k$ as $(r_x,p_x)_{x \in \T_N}$ are distributed according to the homogeneous product measure $\mu_{\chi_t({ q})}^N$\ . 

The function $I_M$ is the Legendre transform of the cumulant-generating function of {${\xi}_{0,M}$}\ : 
\begin{equation} I_M(\mathbf{z},{\eta}(t,{ q}))=\sup_{\mathbf{y} \in \R^{2}} \left\{ \mathbf{y} \cdot \mathbf{z} - \log \mu^N_{\chi_t({ q})} [e^{\mathbf{y} \cdot {{\xi}_{0,M}}}]\right\}\ .\end{equation}
Hence \begin{equation}\liminf_{M \to \infty} I_M(\mathbf{z},{\eta}(t,{ q})) \geqslant \sup_{ \mathbf{y} \in \R^2} \left\{ \mathbf{y} \cdot \mathbf{z} - \log \mu^N_{\chi_t({ q})} [e^{\mathbf{y} \cdot {{\xi}_0}}]\right\}=I(\mathbf{z},{\eta}(t,{ q}))\ ,\end{equation}
where $I(\mathbf{z},{\eta}(t,{ q}))$ is the rate function of  $ \left\{ k^{-1} \sum_{i=1}^k \xi_{i}\right\}_k$ as $(r_y,p_y)_{y}$ are distributed according to the homogeneous product measure $\mu^N_{\chi_t({ q})}$\ . 

It follows, by Fatou's lemma, that \eqref{limsup} is smaller than or equal to \begin{equation} \frac{1}{\alpha} \int_{\T} \sup_{\mathbf{z}} \left\{ \alpha v(t,{ q}) \ \Omega(\mathbf{z},{\eta}(t,{ q})) - I(\mathbf{z},{\eta}(t,{ q}))\right\} d{ q}\ . \end{equation}
From now on we omit the dependance in $(t,{ q})$ of the involved functions $v$ and ${\eta}$. Recall that $\chi$ and ${\eta}$ are in duality (see \eqref{fenchel}). An easy computation gives that \begin{align} I(\mathbf{z},{\eta})& =\sup_{\mathbf{y}} \left\{ \mathbf{y}\cdot \mathbf{z} - \log \left( \int_{{\R^{2}}} e^{\mathbf{y}\cdot \xi} e^{\chi \cdot \xi - \log Z(\chi)} d\r d\p\right)\right\} \notag\\ 
& = \sup_{\mathbf{y}} \{ \mathbf{y} \cdot \mathbf{z} - \log Z(\chi+\mathbf{y})+\log Z(\chi) \} \notag\\
& = \log Z(\chi)\ { + } \ \mathbf{z} \cdot \chi - S(\mathbf{z})\ ,\end{align}
where the last equality follows from the equality between the Fenchel-Legendre transform of $\log Z$ and the function $-S$. We observe that $I({\eta},{\eta})=0$ and  $D_{\mathbf{z}} { I} (\mathbf{z},{\eta})=0$. Furthermore, $I$ is strictly convex in $\mathbf{z}$: \begin{equation}(D^2_{\mathbf{z}}I )(\mathbf{z},{\eta}) = ({D_{\mathbf{z}}^2}\{-S\})(\mathbf{z}) >0\ .\end{equation}
Since $\Omega({\eta},{\eta})=0$ and $(D_{\mathbf{z}}\Omega)(\mathbf{z},{\eta})=(DH)(\mathbf{z})-DH({\eta})$, we also get: $(D_{\mathbf{z}}\Omega)({\eta},{\eta})=0\ .$

\begin{lem} \label{lem:compact}
For $\alpha >0$ sufficiently small, \begin{equation}\forall \ \mathbf{z} \in {\R^{2}},\ \forall \ { q} \in \T, \  \alpha v(t,{ q})  \Omega(\mathbf{z}, {\eta}(t,{ q})) \leqslant I(\mathbf{z},{\eta}(t,{ q}))\ .\end{equation}
\end{lem}

\begin{proof}  { An easy computation provides an explicit expression for the rate function: if $\mathbf{z}=(z_1,z_2)$ and $\eta=(e,r)$ with $ e-r^2/2 >0$ then } \begin{equation} I(\mathbf{z},\eta)= \frac{1}{e-r^2/2}\left(\frac{r^2}{2}-z_2 r+z_1\right)-\log\left(\frac{z_1-z_2^2/2}{e-r^2/2}\right) -1\ .\end{equation}{From the inequality $-\log x \geqslant -x+1$ (satisfied for any $x>0$), we get } \begin{equation}I(\mathbf{z},\eta) \geqslant \frac{1}{2(e-r^2/2)}(r-z_2)^2\ .\end{equation}{Thus, for a given ${\eta}$, the rate function $\mathbf{z} \to I(\mathbf{z},{\eta})$ is such that $I(\mathbf{z},{\eta}) \geqslant c_{{\eta}} \vert \mathbf{z}-{\eta}\vert^2,$ where $c_{{\eta}}$ is a positive constant. Moreover, according to \eqref{eq:beta},} \begin{equation}  \forall \ t \in [0,T],\ \forall \ q \in \T,\ c_{\eta(t,q)} \geqslant c > 0\end{equation}{Let us fix $\mathbf{z} \in \R^2$.} From the Taylor-Lagrange theorem, there exists a positive constant $C$ such that \begin{equation}\Omega(\mathbf{z},{\eta}(t,{ q})) \leqslant C \vert \mathbf{z}-{\eta}(t,{ q})\vert^2 \leqslant I(\mathbf{z},{\eta}(t,{ q}))\ .\end{equation} {More precisely, $C$ is equal to} \begin{equation}\sup_{(t,q)\in [0,T]\times \T} \Vert D^2H(\eta(t,q))\Vert^2\ .\end{equation}
Since $v$ is uniformly bounded, the result is proved. \end{proof}

Consequently, for $\alpha$ small enough, \begin{equation}\sup_{\mathbf{z}} \left\{ \alpha v(t,{ q}) \ \Omega(\mathbf{z},{\eta}(t,{ q}))-I(\mathbf{z},{\eta}(t,{ q}))\right\}=0\ ,\end{equation} and we have finally proved that \begin{equation}\d_t H_N(t) \leqslant C \ H_N(t) + R_{N,\ell,M}(t)\end{equation} with \begin{equation}\lim_{M \to \infty} \lim_{\ell \to \infty} \lim_{N\to \infty} \frac{1}{N} \int_0^t R_{N,\ell,M}(s) \ ds =0\ .\end{equation}

By Gronwall's inequality we obtain: $H_N(t)/N \xrightarrow[N \to \infty]{} 0$ and Theorem \ref{theo:hydro} is proved.

\section{Proof of Theorem \ref{theo:moments}: Moments Bounds}\label{sec:bounds}

In the following, we prove the two conditions on the moments bounds  for a class of local equilibrium states. First, we assume that the initial law $\mu_0^N$ is exactly the Gibbs local equilibrium measure $\mu_{\beta_0(\cdot),\lambda_0(\cdot)}^N$. Second, we extend the proof to the case where $\mu_0^N$ is a convex combination of Gibbs local equilibrium measures.

We need to control the moments $  \mu_t^N \left[\sum_{x } e_x^{k} \right]$ for all $k \geqslant 1$. The { first two} bounds ($k=1,2)$ would be sufficient to justify the cut-off of the currents, but here we need more  because of Lemma \ref{lem:sym} (which is necessary to prove Proposition \ref{prop}). {Since the chain is harmonic, Gibbs states are Gaussian. We recall that all Gaussian moments can be expressed in terms of variances and covariances}. In the following,  we first give an other representation of the dynamics of the process, and then we prove the bounds and { precise their dependence on $k$}.

Let us highlight that, from now on, we consider the process with generator $\L_N$: it is not accelerated any more. The law of this new process $(\tilde{\omega}_t)_{t \geqslant 0}$ is denoted by $\tilde{\mu}_t^N$. At the end of this part, Theorem \ref{theo:moments} will be easily deduced since all estimates will not depend on $t$, and the following equality still holds:\begin{equation}\mu_t^N=\tilde{\mu}_{tN^2}^N\ .\end{equation}

\textbf{Remarks.} \begin{enumerate} \item In the following, we always respect the decomposition of the space $\Omega^N=\R^N \times \R^N$. Let us recall that the first $N$ components stand for $\r$ and the last $N$ components stand for $\p$. All vectors and matrices are written according to this decomposition. Let $\nu$ be a measure on $\Omega^N$. We denote by $m \in \R^{2N}$ its mean vector and by $C \in \mathfrak{M}_{2N}(\R)$ its { correlation} matrix (see \eqref{cor}). We can write $m$ and $C$ as 
 \begin{equation}m=(\rho,\pi) \in \R^{2N} \ \text{ and } \ C=\begin{pmatrix} U & Z^{\ast} \\ Z & V \end{pmatrix} \in \mathfrak{S}_{2N}(\R)\ ,\end{equation} where $\rho:=\nu[\r] \in \R^N $ , $\pi :=\nu[\p] \in \R^{N}$ and $U, V, Z \in \mathfrak{M}_{N}(\R)$\ .
 
 \item Thanks to the convexity inequality $(a+b)^k \leqslant 2^{k-1} \ (a^k+b^k)$, for $a,b>0$, we can write \begin{equation}e_x^k \leqslant \frac{1}{2} \ \left(p_x^{2k} + r_x^{2k} \right)\ .\end{equation} Thus, instead of proving \eqref{mom} we will show \begin{equation}\mu_t^N \left[\sum_{x \in \T_N} p_x^{2k} \right] \leqslant (C k)^k \ N \quad \text{and} \quad \mu_t^N \left[\sum_{x \in \T_N} r_x^{2k} \right] \leqslant (Ck)^k \ N\ .\end{equation}
 
 \end{enumerate}

\subsection{Poisson Process and Gaussian Measures}\label{subsec:process}

We are going to use a graphical representation of the process $(\tilde{\omega}_t)_{t \geqslant 0}$.

Let us define \begin{equation} A:=\begin{pmatrix} 
0 & \cdots  & \cdots & 0 & -1 & 1 &  & (0) \\
\vdots & & & \vdots & 0 & \ddots & \ddots &   \\
\vdots &  & & \vdots &  0 & & \ddots  & 1 \\
0 & \cdots  & \cdots & 0 & 1 & 0 & 0  & -1 \\
1 & 0  & 0 & -1 & 0 & \cdots & \cdots  & 0\\
-1 & \ddots  & & 0 & \vdots & & &  \vdots  \\
  & \ddots  & \ddots & 0 & \vdots & & &  \vdots \\
 (0) &   & -1 & 1 & 0 & \cdots & \cdots &  0 \end{pmatrix} \in \mathfrak{M}_{2N}(\R)\ . \label{matrix} \end{equation}
 
We now consider $(m_t,C_t)_{t \geqslant 0}$, a Markov process on $\R^{2N} \times \mathfrak{S}_{2N}(\R)$ whose generator is denoted by $\mathcal{G}$ and defined as follows.

Take $m:=(\rho,\pi) \in \R^{2N}$ and $C:=\begin{pmatrix} U & Z^\ast \\ Z & V \end{pmatrix} \in \mathfrak{S}_{2N}(\R)$, where $\rho, \pi$ are two vectors in $ \R^N$, $U, V$ are two symmetric matrices in $\mathfrak{S}_N(\R)$ and  $Z$ is a matrix in $\mathfrak{M}_N(\R)$. Hereafter, we denote by $Z^\ast$ the transpose of the matrix $Z$.

The generator $\mathcal{G}_N$ is given by \begin{equation}(\mathcal{G}_Nv)(m,C):=(\mathcal{K}_Nv)(m,C)+ \gamma \ (\mathcal{H}_Nv)(m,C)\ , \label{markov} \end{equation}
where \begin{equation}\mathcal{K}_N:=\sum_{i,j\in \T_N} (-AC+CA)_{i,j} \  \d_{C_{i,j}} + \sum_{i\in \T_N} { \left\{ (\pi_{i+1}-\pi_i) \d_{\rho_i} + (\rho_i-\rho_{i-1}) \d_{\pi_i}\right\}}\ ,\end{equation}
and \begin{equation}(\mathcal{H}_Nv)(m,C):=\frac{1}{2}\sum_{k\in\T_N} [v(m^k,C^k)-v(m,C)]\ .\end{equation} Here, \begin{equation}m^k=(\rho,\pi^k) \ \text{ and } \ C^k=\Sigma_k^{\ast} \cdot C \cdot \Sigma_k = \begin{pmatrix} U & Z^{k \ast} \\ Z^k & V^k \end{pmatrix}\ .\end{equation}
{In these last two formulas,  $\pi^k$ is the vector obtained from $\pi$ by the flip of $\pi_k$ into $-\pi_k$, and $\Sigma_k$ is defined as }\begin{equation}\Sigma_k=\begin{pmatrix} I_n & 0_n \\ 0_n & I_n-2E_{k,k} \end{pmatrix}.\end{equation}
More precisely,
\begin{equation}Z_{i,j}^k=(-1)^{\delta_{k,i}} Z_{i,j} \ \text{ and } \ V_{i,j}^k=(-1)^{(\delta_{k,i}+\delta_{k,j})}V_{i,j}\ .\end{equation} 

We denote by $\P_{m_0,C_0}$ the law of the process $(m_t,C_t)_{t \geqslant 0}$ starting from $(m_0,C_0)$, and by $\mathbb{E}_{m_0,C_0}[\cdot]$ the expectation with respect to $\P_{m_0,C_0}$. 

For $t \geqslant 0$ fixed, let ${\theta}^t_{m_0,C_0}(\cdot,\cdot)$ be  the law of the random variable $(m_t,C_t)\in \R^{2N}\times \mathfrak{S}_{2N}(\R)$, knowing that the process starts from $(m_0,C_0)$.

Recall that we denote by $G_{m,C}(\cdot)$ the Gaussian measure on $\Omega^N$ with mean $m \in \R^{2N}$ and { correlation} matrix $C\in \mathfrak{S}_{2N}(\R)$.

\begin{lem}\label{lem:pp} Let $\mu_0^N:=\mu_{\beta_0(\cdot),\lambda_0(\cdot)}^N$ be the Gibbs equilibrium state defined by \eqref{gibbs}, where $\lambda_0(\cdot)$ and $\beta_0(\cdot)$ are the two macroscopic potential profiles.

Then, \begin{equation} \tilde{\mu}_t^N = \int G_{m,C}(\cdot) \ d\theta_{m_0,C_0}^t(m,C) \label{process} \end{equation}

where  \begin{equation} m_0:=\left(-\frac{\lambda_0}{\beta_0}\left(\frac{0}{N}\right),\cdots,-\frac{\lambda_0}{\beta_0}\left(\frac{N-1}{N}\right), \underbrace{0, \cdots, 0}_N\right) \end{equation} and
 \begin{equation} C_0:=\begin{pmatrix} D & 0_N \\ 0_N & D' \end{pmatrix}\qquad \text{ with } \left\{ \begin{aligned} D & =\text{\emph{diag}}\left(\cdots, \frac{1}{\beta_0(x/N)} + \frac{\lambda_0^2(x/N)}{\beta_0^2(x/N)}, \cdots \right), \\
D' & =\text{\emph{diag}}\left(\cdots, \frac{1}{\beta_0(x/N)}, \cdots \right).\end{aligned} \right.  \end{equation}

\end{lem}

\begin{proof} { We begin with the graphical representation of the process $(\tilde{\omega}_t)_{t \geqslant 0}$, which is based on the Harris description. Let $(N_i)_{i\in\T_N}$ be a sequence of independent standard Poisson processes of intensity $\gamma$.} In other words, we put on each site $i \in \T_N$ an exponential clock of mean $1/\gamma$. At time 0 the process has an initial state $\omega_0$. Let $T_1=\text{inf}_{t \geqslant 0} \left\{\exists \ i \in \T_N,\ N_i(t)=1 \right\}$ and $i_1$ the site where the infimum is achieved.

During the interval $[0,T_1)$, the process follows the deterministic evolution given by the generator $\A_N$. More precisely, let $F: (\r,\p) \in \T_N^2 \to A \cdot (\r, \p) \in \T_N^2$ where $A$ is given by \eqref{matrix}. Then, for any continuously differentiable function $f: \Omega^N \to \R$, \begin{equation}\A_N f(\omega)= A \cdot D f (\omega)\ ,\end{equation} and during the time interval $[0,T_1)$, $\tilde{\omega}_t$ follows the evolution given by the system: $dy / dt=F(y).$
At time $T_1$, the momentum $p_{i_1}$ is flipped, and gives a new configuration. Then, the system starts again with the deterministic evolution up to the time of the next flip, and so on. Let $\xi:=(i_1, T_1), \dots, (i_k,T_k), \dots$ be the sequence of sites and ordered times for which we have a flip, and let us denote its law by $\mathbb{P}$. Conditionally to $\xi$, the evolution is deterministic, and the state of the process $\tilde{\omega}^{\xi}_t$ is given by \begin{equation} \forall \ t \in [T_k,T_{k+1}),\ \tilde{\omega}^{\xi}_t=e^{(t-T_k)A} \circ F_{i_k} \circ e^{(T_k-T_{k-1})A} \circ F_{i_{k-1}} \circ \cdots \circ e^{T_1 A} \omega_0\ , \label{omega} \end{equation}
where $F_i$ is the map $\omega = (\r,\p) \to (\r,\p^i)$.

If initially the process starts from $\omega_0$ which is distributed according to a Gaussian measure $\mu_0^N$, then $\tilde{\omega}^{\xi}_t$ is distributed according to a Gaussian measure $\tilde{\mu}^{\xi}_t$. Then,  the density $\tilde{\mu}_t^N$ is  given by \begin{equation} \tilde{\mu}^N_t(\cdot)=\int \tilde{\mu}^{\xi}_t(\cdot) \ d\P(\xi)\ .\label{law1} \end{equation}

{More precisely, the mean vector $m_t^\xi$ and the correlation matrix $C_t^\xi$ of $\tilde\mu_t^\xi$ can be related to the mean vector $m_0$ and the correlation matrix $C_0$ of $\mu_0^N$: } \begin{equation}  m_t^{\xi}=e^{(t-T_k)A}\cdot \Sigma_{i_k} \cdot e^{(T_k-T_{k-1})A} \cdot \Sigma_{i_{k-1}} \cdots e^{T_1A} \cdot m_0\ , \label{moypoiss} \end{equation}
{ and}
\begin{equation}  C^{\xi}_t=e^{(t-T_k)A} \cdot \Sigma_{i_k} \cdot e^{(T_k-T_{k-1})A}  \cdots \Sigma_{i_1} \cdot e^{T_1A} \cdot C_0 \cdot e^{-T_1A} \cdot \Sigma_{i_1}^\ast \cdots   e^{-(T_k-T_{k-1})A}  \cdot \Sigma_{i_{k}}^\ast e^{-(t-T_k)A}\ .\label{corrpoiss} \end{equation} {Equations \eqref{moypoiss} and \eqref{corrpoiss} also give a graphical representation of the process $(m_t,C_t)_{t\geqslant 0}$: during the interval $[0,T_1)$, $m_t$ follows the evolution given by the (vectorial) system \begin{equation}\frac{dy}{dt}=F(y)\end{equation} (where $F$ has been previously  introduced  for the process $\tilde \omega_t$). At time $T_1$, the component $m_{i_1+N}$ (which corresponds to the mean of $p_{i_1}$) is flipped, and gives a new mean vector. Then, the deterministic evolution goes on up to the time of the next flip, and so on. 

In the same way, during the interval $[0,T_1)$, $C_t$ follows the evolution given by the (matrix) system: \begin{equation}\frac{dM}{dt}=-AM+MA\end{equation} (where $A$ has been previously defined). At time $T_1$, all the components $C_{i_1,j}$ and $C_{i,i_1}$ when $j \neq i_1$ and $i\neq i_1$ are flipped and the matrix $C_{T_1}$ becomes $\Sigma_{i_1} \cdot C_{T_1} \cdot \Sigma_{i_1}^*$\ .  The generator of this Markov process $(m_t,C_t)_{t \geqslant 0}$ is exactly the one defined by \eqref{markov}. Consequently, for $t \geqslant 0$, the law of the random variable $(m_t,C_t)$ is $\theta_{m_0,C_0}^t$, where}  \begin{equation} m_0=\left(-\frac{\lambda_0}{\beta_0}\left(\frac{0}{N}\right),\cdots,-\frac{\lambda_0}{\beta_0}\left(\frac{N-1}{N}\right), \underbrace{0, \cdots, 0}_N\right) \end{equation}
and 
 \begin{equation} C_0=\begin{pmatrix} D & 0_N \\ 0_N & D' \end{pmatrix}\  \text{ where } \left\{ \begin{aligned} D & =\text{diag}\left(\cdots, \frac{1}{\beta_0(x/N)} + \frac{\lambda_0^2(x/N)}{\beta_0^2(x/N)}, \cdots \right)\ , \\
D' & =\text{diag}\left(\cdots, \frac{1}{\beta_0(x/N)}, \cdots \right)\ ,\end{aligned} \right.  \end{equation} as it can be deduced from Lemma \ref{lem:matrix}.  Recall that in this section, $\mu_0^N$ is given by 
\begin{equation}\mu_0^N(d\r,d\p)=\prod_{x \in \T_N} \frac{\exp\left(-\beta_0\left(x/N\right) e_x-\lambda_0(x/N  )r_x\right) }{Z(\beta_0(\cdot),\lambda_0(\cdot))} dr_x dp_x\ .\end{equation}

It follows that  the density $\tilde{\mu}_t^N$ is  equal to  \begin{equation} \tilde{\mu}^N_t(\cdot)=\int \tilde{\mu}^{\xi}_t(\cdot) \ d\P(\xi)=\int G_{m,C}(\cdot) \ d\theta_{m_0,C_0}^t(m,C)\ .\label{law} \end{equation}

 \end{proof}

\textbf{Remark.} Observe that \begin{equation}\tilde{\mu}_t^N[p_x] = \int G_{m,C}(p_x) \ d\theta_{m_0,C_0}^t(m,C)=\int \pi_x \ d\theta_{m_0,C_0}^t(m,C)=\mathbb{E}_{m_0,C_0} [\pi_x(t)]\ ,\end{equation}
\begin{equation}\tilde{\mu}_t^N[r_x] = \int G_{m,C}(r_x) \ d\theta_{m_0,C_0}^t(m,C)=\int \rho_x \ d\theta_{m_0,C_0}^t(m,C)=\mathbb{E}_{m_0,C_0} [\rho_x(t)]\ .\end{equation}

\begin{lem}\label{inequ} Let $(m_t,C_t)_{t \geqslant 0}$ be the Markov process defined above. As previously done, we introduce $\rho(t), \pi(t) \in \R^N$ and $U(t), V(t), Z(t) \in \mathfrak{M}_N(\R)$ such that \begin{equation}m_t=(\rho(t),\pi(t)) \ \text{ and } \ C_t=\begin{pmatrix} U(t) & Z^\ast(t) \\ Z(t) & V(t) \end{pmatrix}\end{equation} Then,
\begin{equation}\P_{m_0,C_0} \text{ - a. s. }, \ \forall \ t \geqslant 0, \ \left\{ \begin{aligned}\pi_y^2(t) & \leqslant V_{y,y}(t)\ , \\ \rho_y^2(t) & \leqslant U_{y,y}(t)\ . \end{aligned} \right.\end{equation}

\end{lem}

\begin{proof} First of all, let us notice that the quantities $V_{y,y}(t)-\pi_y^2(t)$ and $U_{y,y}(t)-\rho_y^2(t)$ are the diagonal components of the symmetric matrix $S_t:=m_t\cdot \! ^tm_t - C_t$. From Lemma \ref{lem:pp}, we have \begin{equation}S_t=\int S_t^\xi \ d\P(\xi) \ . \end{equation} 

For any sequence of sites and ordered times  $\xi=(i_1,T_1),\dots,(i_k,T_k),\dots$, the symmetric matrix $S_t^\xi$ is positive because this is the matrix of covariances of $\tilde\omega_t^\xi$. It follows that $S_t$ is positive, and its diagonal components are all positive. \end{proof}

\textbf{Remark.} In the case of the pinned chain, the matrix $A$ is slightly different, but all the notations and conclusions are still valid. The initial { correlation} matrix for the pinned model is not more diagonal, but has non-trivial values on the upper and lower diagonals. The initial mean vector is equal to $0_{\R^{2N}}$.

\subsection{The Evolution of $(m_t, C_t)_{t \geqslant 0}$}\label{subsec:evolution}

Thanks to the regularity of $\beta_0$ and $\lambda_0$, we know that there exists a constant $K$ which does not depend on $N$ such that 
\begin{equation} \left\{ \begin{aligned} & \frac{1}{N} \sum_{i,j} \left[ (U_{i,j})^2(0) + (V_{i,j})^2(0) + 2 (Z_{i,j})^2(0)  \right] \leqslant K \ ,\\ & \frac{1}{N} \sum_{i}\left[U_{i,i}(0) + V_{i,i}(0)  \right] \leqslant K\ , \\
& \frac{1}{N} \sum_i \left[(U_{i,i})^k(0) + (V_{i,i})^k(0)\right]  \leqslant K^k\ , \ \text{for all } k\geqslant 1 \ . \end{aligned} \right. \label{hyp} \end{equation}
Moreover, one can easily show that \begin{equation} \mathcal{G}\left(\sum_{i,j} (U_{i,j})^2+(V_{i,j})^2+2(Z_{i,j})^2\right)=0 \quad \text{ and } \quad \mathcal{G}\left(\sum_{i} U_{i,i}+ V_{i,i}\right)=0\ .\end{equation}
It results that the two first inequalities of \eqref{hyp} are actually uniform in $t$, in the sense that 
\begin{equation} \left\{ \begin{aligned} & \frac{1}{N} \mathbb{E}_{m_0,C_0} \left[\sum_{i,j} \left[ (U_{i,j}(t))^2 + (V_{i,j}(t))^2 + 2 (Z_{i,j}(t))^2 \right]\right] \leqslant K\ , \\ & \frac{1}{N} \mathbb{E}_{m_0,C_0} \left[\sum_{i}\left[U_{i,i}(t)+ V_{i,i}(t) \right]\right] \leqslant K\ .\end{aligned} \right. \label{unif}\end{equation}
We are going to see how this last inequality can be used in order to show \eqref{mom}. We denote by $u_k(t)$ and $v_k(t)$ the two quantities \begin{equation}\left\{ \begin{aligned} u_k(t) & = \mathbb{E}_{m_0,C_0} \left[ \sum_{i \in \T_N} U_{i,i}^k(t)  \right]\ , \\ v_k(t) & = \mathbb{E}_{m_0,C_0}\left[\sum_{i \in \T_N} V_{i,i}^k(t) \right]\ . \end{aligned} \right.\end{equation}
Let us make the link with \eqref{mom}. In view of \eqref{law}, we can write \begin{equation}\tilde{\mu}_t^N \left[p_y^{2k}\right] = \int G_{m,C}\left[p_y^{2k}\right] \ d\theta^t_{m_0,C_0}(m,C)\ ,\end{equation}
\begin{equation}\tilde{\mu}_t^N \left[r_y^{2k}\right] = \int G_{m,C}\left[r_y^{2k}\right] \ d\theta^t_{m_0,C_0}(m,C)\ .\end{equation}
We use the convexity inequality $(a+b)^{2k} \leqslant 2^{2k-1} \ (a^{2k} + b^{2k} )$ - which is true for all $a,b \in \R$ - to get \begin{align} \tilde{\mu}_t^N \left[p_y^{2k}\right]& = \int G_{m,C}\left[(p_y-\pi_y+\pi_y)^{2k}\right] \ d\theta^t_{m_0,C_0}(m,C) \notag\\
& \leqslant 2^{2k-1} \int G_{m,C} \left[ (p_y - \pi_y)^{2k} \right] \ d\theta^t_{m_0,C_0}(m,C) + 2^{2k-1} \int \pi_y^{2k} \ d\theta_{m_0,C_0}^t(m,C)\ .\end{align}
We deal with the two terms of the sum, separately. First, observe that Gaussian centered moments are easily computable: \begin{equation}G_{m,C}\left[ (p_y-\pi_y)^{2k}\right] = \left(V_{y,y}-\pi_y^2\right)^k \frac{(2k)!}{k! \ 2^k}\ .\end{equation}
Then, \begin{equation}\sum_{y \in \T_N} \int \left(V_{y,y}-\pi_y^2\right)^k \frac{(2k)!}{k! \ 2^k} \ d\theta^t_{m_0,C_0}(m,C) \leqslant \frac{(2k)!}{k! \ 2^k} \left( v_k(t) + \mathbb{E}_{m_0,C_0} \left[\sum_{y \in \T_N} \pi_y^{2k}(t)\right] \right)\ .\end{equation}
In the same way, \begin{equation}\sum_{y \in \T_N} \int \left(U_{y,y}-\rho_y^2\right)^k \frac{(2k)!}{k! \ 2^k} \ d\theta^t_{m_0,C_0}(m,C) \leqslant \frac{(2k)!}{k! \ 2^k} \left( u_k(t) + \mathbb{E}_{m_0,C_0} \left[\sum_{y \in \T_N} \rho_y^{2k}(t)\right] \right)\ .\end{equation}
Lemma \ref{inequ} shows that  \begin{equation}\begin{aligned} \mathbb{E}_{m_0,C_0} \left[\sum_{y \in \T_N} \pi_y^{2k}(t)\right] & \leqslant \mathbb{E}_{m_0,C_0} \left[ \sum_{y \in \T_N} V_{y,y}^k(t)\right]=v_k(t)\ , \\  \mathbb{E}_{m_0,C_0} \left[\sum_{y \in \T_N} \rho_y^{2k}(t)\right] & \leqslant \mathbb{E}_{m_0,C_0} \left[ \sum_{y \in \T_N} U_{y,y}^k(t)\right]=u_k(t)\ .\end{aligned}  \end{equation}
As a result, \begin{equation}\sum_{y} \tilde{\mu}_t^N \left[p_y^{2k}\right] \leqslant \frac{(2k)!}{k!} \ v_k(t) \sim 2 \left(\frac{4}{e}\right)^k  k^k \ v_k(t)\ ,\end{equation} \begin{equation}\sum_{y} \tilde{\mu}_t^N \left[r_y^{2k}\right] \leqslant \frac{(2k)!}{k!} \ u_k(t)\sim 2 \left(\frac{4}{e}\right)^k  k^k \ u_k(t)\ .\end{equation}
In a few words, to get \eqref{mom}, we need to estimate the two quantities $u_k(t)$ and $v_k(t)$, which are  related to  $C_t$. That is what we do in the next section. 

\textbf{Remark.} In the case of the pinned model, the  $p_x$ and $q_x$  remain centered during the evolution: for all $t>0$, $m_t=0_{\R^{2N}}\ . $ This simplifies the study since we do not need to center the variables. The result is the same: we need to estimate $u_k(t)$ and $v_k(t)$.

\subsection{The { Correlation} Matrix}\label{ss:cor}

\begin{lem} \label{lem:cov} For { any integer $k$ not equal to 0}, there exists  a positive constant  $K$ which does not depend on $N$ and $t$ such that \begin{equation} \left\{ \begin{aligned} & v_k(t) \leqslant K^k \ N\ , \\  & u_k(t)  \leqslant K^k \ N\ .  \end{aligned}\right. \end{equation} 

\end{lem}

\begin{proof} First of all,  { \eqref{unif}}  shows that, uniformly in $t$, \begin{equation} \left\{ \begin{aligned} u_1(t) &\leqslant K N \\ u_2(t) & \leqslant K N\end{aligned}\right. \ \text{ and } \ \left\{ \begin{aligned} v_1(t) & \leqslant KN \\ v_2(t) & \leqslant KN\ . \end{aligned} \right.\end{equation} We observe that \begin{equation} u_k(t) + v_k(t) =\mathbb{E}_{m_0,C_0} \left[ \sum_{i \in \T_N} C_{i,i}^k(t)\right] = \int \sum_{i \in \T_N} (C_{i,i}^\xi)^k(t) \ d\P(\xi)\ .\end{equation}

Thanks to the dynamics description, we know the expression of the { correlation} matrix: conditionally to $\xi$,  for all $t \in [T_k,T_{k+1})$, 
\begin{equation}C^{\xi}(t)=e^{(t-T_k)A} \cdot \Sigma_{i_k} \cdot e^{(T_k-T_{k-1})A} \cdots \Sigma_{i_1} \cdot e^{T_1A} \cdot C_0 \cdot e^{-T_1A} \cdot \Sigma_{i_1}^\ast \cdots   e^{-(T_k-T_{k-1})A}  \cdot \Sigma_{i_{k}}^\ast e^{-(t-T_k)A}\ ,  \end{equation}
Consequently, since $C_0$ and $C^{\xi}(t)$ are similar, we have: \begin{equation}\forall \  k \in \N,\ \text{Tr}([C^{\xi}(t)]^k)=\text{Tr}(C^k_0)=O(N)\ . \end{equation}
More precisely, \begin{equation} \text{Tr}(C^k_0)=\sum_{i \in \T_N} U_{i,i}^k(0)+ V_{i,i}^k(0) = \sum_{i \in \T_N} \frac{1}{\beta_0^k(i/N)} + \left(\frac{1}{\beta_0(i/N)} + \frac{\lambda^2_0(i/N)}{\beta^2_0(i/N)}\right)^k\ .\end{equation} From  \eqref{hyp}  we get $\text{Tr}(C^k_0) \leqslant N  K^k$, where $K$ does not depend on $N$, $\xi$ and $t$: \begin{equation}K:= \sup_{u \in [0,1]} \left\{ \frac{1}{\beta_0(u)} + \frac{\lambda_0^2(u)}{\beta_0^2(u)}\right\}\ .\end{equation} Now we show that the same inequality holds for $ \sum_i [C_{i,i}^\xi]^k(t)$. The matrix $C^\xi(t)$ is symmetric, hence diagonalizable, and after denoting its eigenvalues by $\lambda_1, ..., \lambda_{2N}$, we can write \begin{equation}\text{Tr}([C^\xi(t)]^k)=\sum_{i=1}^{2N} \lambda_i^k\ .\end{equation}
We have now to compare $  \sum_{i=1}^{2N} \lambda_i^k$ with $\sum_{i=1}^{2N} [C_{i,i}^\xi]^k(t)$. But, if we denote by $P$ the orthogonal matrix of the eigenvectors of $C^\xi(t)$, then we get $C^\xi(t)=(P_t^\xi)^{\ast} \cdot D \cdot P_t^\xi$, where $D$ is the diagonal matrix with the eigenvalues $\lambda_1,..., \lambda_{2N}$. For the sake of simplicity, we denote by $(P_{i,j})_{i,j}$ the components of $P_t^\xi$. Then, \begin{equation} [C^\xi_{i,i}]^k(t)  =\left( \sum_{j,l} P^{\ast}_{i,j} D_{j,l} P_{l,i}\right)^k  = \left( \sum_j P_{i,j}^{\ast} \lambda_j P_{j,i} \right)^k = \left(\sum_{j} P_{i,j}^{\ast} P_{j,i} \cdot \lambda_j\right)^k.\end{equation}
But, $ \sum_{j} P_{i,j}^{\ast} P_{j,i}=1$, since $D$ is an orthogonal matrix. 
Consequently, we can use the convexity inequality, and we obtain 
\begin{equation} \sum_{i } [C^\xi_{i,i}]^k (t) \leqslant \sum_i \sum_j P_{i,j}^{\ast} P_{j,i}  \lambda_j^k \leqslant \sum_j \lambda_j^k = \text{Tr}([C^\xi(t)]^k) \leqslant N  K^k\ .\end{equation}
Hence, \begin{equation}u_k(t)+v_k(t) \leqslant \int N K^k \ d\P(\xi) \leqslant N K^k\ .\end{equation} \end{proof}

\textbf{Remark.} We notice that the same proof works for the pinned case. The only difference is about the initial matrix $C_0$, but the smoothness of the profile $\beta_0$ is still true, and the estimate $\text{Tr}(C_0^k)=O(N)$ is valid.

\subsection{When $\mu_0^N$ is a Convex Combination of Gibbs Measures}\label{subsec:convex}

As in Theorem \ref{theo:moments}, we now suppose that the initial probability measure $\mu_0^N$ is a convex combination of Gibbs states defined by \begin{equation} \mu_0^N(\cdot)=\int G_{m_0,C_0}(\cdot) \ d\sigma(m_0,C_0)\ . \label{init} \end{equation} If initially the process starts from $\omega_0$ which is distributed according to a Gaussian measure $G_{m_0,C_0}$, we know from Lemma \ref{lem:pp} that $\tilde{\omega}_t$ is distributed according to a convex combination of Gaussian measures written as \begin{equation}\int G_{m,C}(\cdot) \ d\theta^t_{m_0,C_0}(m,C)\ .\end{equation} 
Consequently, in the case where $\mu_0^N$ is given by \eqref{init}, the law of the process $\tilde{\omega}_t$ is given by \begin{equation}\tilde{\mu}_t^N(\cdot)= \int \left\{ \int  G_{m,C}(\cdot) \ d\theta^t_{m_0,C_0}(m,C) \right\} \ d\sigma(m_0,C_0)\ .\end{equation}
Let us recall that we want to control, for $k \geqslant 1$, $ \tilde{\mu}_t^N \left[\sum_{x \in \T_N} p_x^{2k}\right]$ and $\tilde{\mu}_t^N\left[\sum_{x \in \T_N} r_x^{2k} \right].$
Following the lines of the previous section, we notice that it is sufficient to control two quantities: \begin{equation}\left\{ \begin{aligned} & \int  \mathbb{E}_{m_0,C_0} \left[\sum_{i \in \T_N} U_{i,i}^k(t) \right]  d\sigma(m_0,C_0)\ , \\ & \int \mathbb{E}_{m_0,C_0} \left[\sum_{i \in \T_N} V_{i,i}^k(t) \right]  d\sigma(m_0,C_0)\ . \end{aligned} \right. \end{equation} 
Lemma \ref{lem:cov} gives a constant $C(\lambda_0,\beta_0)$ which does not depend on $N$ and $t$ such that \begin{equation}\left\{ \begin{aligned}\mathbb{E}_{m_0,C_0} \left[\sum_{i \in \T_N} U_{i,i}^k (t) \right]  & \leqslant [C(\lambda_0,\beta_0)]^k \ N\ , \\ \mathbb{E}_{m_0,C_0} \left[\sum_{i \in \T_N} V_{i,i}^k (t) \right] & \leqslant [C(\lambda_0,\beta_0)]^k \ N\ . \end{aligned} \right.\end{equation}
More precisely, \begin{equation}C(\lambda_0,\beta_0)= \sup_{u \in [0,1]} \left\{ \frac{1}{\beta_0(u)} + \frac{\lambda_0^2(u)}{\beta_0^2(u)}\right\}\ . \end{equation}

In order to keep the same  control, we have to suppose that, for all $k \geqslant 1$, \begin{equation}\int [K(m,C)]^k \ d \sigma(m,C) < \infty, \ \text{ where } K(m,C):=\sup_{i \in \T_N} C_{i,i}\ . \end{equation}
Finally, let us observe that all estimates are given for $\tilde{\mu}_t^N$ but are still true for the accelerated law $\mu_t^N$. Indeed, the constants that appear do not depend on $N$ and $t$.

\appendix 

\section{Proof of the Taylor Expansions}\label{appa}

Now we prove Proposition \ref{prop}. For the sake of simplicity, we define \begin{equation}\left\{\begin{aligned} 
g_x(\r,\p) & :=-\frac{r_x}{2\gamma}(p_{x+1}+p_x+\frac{\gamma}{2}r_x)\ , \\
f_x(\r,\p) &:=-\frac{p_{x+1}}{\gamma}\ , \\
\delta_x(\r,\p) &:=\beta'_t\left(\frac{x}{N}\right) \ g_x +  \lambda'_t\left(\frac{x}{N}\right) \ f_x =  F\left(t,\frac{x}{N}\right) \cdot \tau_xh(\r,\p)\ .\end{aligned} \right. \end{equation}
First we will compute the first part that appears in the integral $ N^2 \left(\phi_t^N\right)^{-1} \L_N^* \phi_t^N$ , then we will compute the second part $\ds -\d_t \phi_t^N/ \phi_t^N \times f_t^N$.

\subsection{First Term: the Adjoint Operator}\label{subsec:adjoint}

\begin{lem} 
\begin{align}
\A\phi_t^N =& \frac{\phi_t^N}{N^2} \sum_{x \in \T_N} \beta''_t\left(\frac{x}{N}\right) \left[p_{x+1}r_x+ \frac{p_x^2+r_xr_{x-1}}{2 \gamma}\right] -\lambda''_t\left(\frac{x}{N}\right)\left[p_{x+1}+\frac{r_{x+1}}{\gamma}\right] \notag\\
& + \frac{\phi_t^N}{N^2} \sum_{x \in \T_N} \left[\L^*(\delta_x) + \A(\delta_x)\right]  + o\left(\frac{1}{N}\right).\end{align}

\end{lem}

\begin{proof} First, remind that the expression of $\phi_t^N$ is given by  \begin{equation}\phi_t^N(\r,\p)=\frac{(Z(1,0))^n}{Z(\chi_t(\cdot))} \prod_{x\in \T_N} \exp\left(e_x\left(-\beta_t\left(\frac{x}{N}\right)+1\right)-\lambda_t\left(\frac{x}{N}\right)r_x+\frac{1}{N} F\left(t,\frac{x}{N}\right) \cdot \tau_xh(\r,\p)\right).\end{equation}
By definition, \begin{equation}\A \phi_t^N =  \phi_t^N \sum_{x \in \T_N} \left[\left(1-\beta_t\left(\frac{x}{N}\right)\right)\A(e_x)-\lambda_t\left(\frac{x}{N}\right)\A(r_x)\right] + \frac{\phi_t^N}{N} \sum_{x \in \T_N} \A(\delta_x)\ .\end{equation}
We write down the two conservation laws: \begin{align}\mathcal{A}(e_x)& =j_{x+1}^e-j_x^e \quad \text{where } j_x^e:=p_{x}r_{x-1}\ , \\ \mathcal{A}(r_x)& =j_{x+1}^r-j_x^r \quad \text{where } j_x^r:=p_x\ .\end{align}
 
 Hence, \begin{equation}\A \phi_t^N =  \phi_t^N \sum_{x \in \T_N} \left[\left(1-\beta_t\left(\frac{x}{N}\right)\right)\nabla(j_x^e)_x-\lambda_t\left(\frac{x}{N}\right)\nabla(j_x^r)_x\right] + \frac{\phi_t^N}{N} \sum_{x \in \T_N} \A(\delta_x)\ .  \end{equation} where $\nabla(f)_x=f_{x+1}-f_x$\ .

We are interesting in the first two terms in the sum, and we compute a discrete summation by part. Indeed, \begin{equation}\sum_{y \in \T_N} f_y \nabla(g)_y=-\sum_{y \in \T_N} g_{y+1} \nabla(f)_y\  .\end{equation}
We obtain the following terms: \begin{align}
\beta_t\left(\frac{x+1}{N}\right)-\beta_t\left(\frac{x}{N}\right)&=\beta_t'\left(\frac{x}{N}\right)  \frac{1}{N}+\beta_t''\left(\frac{x}{N}\right)  \frac{1}{N^2} + O\left(\frac{1}{N^3}\right),\\
\lambda_t\left(\frac{x+1}{N}\right)-\lambda_t\left(\frac{x}{N}\right)&=\lambda_t'\left(\frac{x}{N}\right) \frac{1}{N}+\lambda_t''\left(\frac{x}{N}\right) \frac{1}{N^2}+O\left(\frac{1}{N^3}\right).\end{align}
First of all, we look at the term obtained in the sum with $ O\left(N^{-3}\right)$. We want to prove \begin{equation}N^2 \int \sum_{x \in \T_N} p_{x+1}r_x \ O\left(\frac{1}{N^3}\right)  f_t^N d\mu_{1,0}^N \leqslant C \ H_N(t) + o(N)\ .\end{equation}
We use the entropy inequality. Let $\varepsilon: \N \to \R$ be a bounded function. We get \begin{equation}\frac{1}{N}\int \sum_{x \in \T_N} p_{x+1}r_x \ \varepsilon(N) f_t^N d\mu_{1,0}^N \leqslant \frac{H_N(t)}{\alpha} + \frac{1}{\alpha} \log \int \exp\left(\frac{\alpha}{N} \sum_{x}p_{x+1}r_x \ \varepsilon(N)\right) \phi_t^N d\mu_{1,0}^N\ .\end{equation}
But, let us recall the inequality $p_{x+1}r_x \leqslant (p_{x+1}^2+r_x^2)/2$ and for $N$ { large} enough, we have 
\begin{equation}\nu_{\chi_t(\cdot)}^N\left[\exp\left(\frac{\alpha}{N}p^2_x \varepsilon(N)\right)\right] \sim_{N \to \infty} \sqrt{\frac{2\pi\ 2N}{N \beta- 2\alpha \varepsilon(N)}} \times \sqrt{\frac{\beta}{2\pi}}  = O(1)\ .\end{equation}
We obtain a similar estimate for $\nu_{\chi_t(\cdot)}^N\left[\exp\left({\alpha}{N}^{-1}r_x \varepsilon(N)\right)\right]$\ .

Therefore, we have showed 
\begin{equation}\frac{1}{N}\int \sum_{x \in \T_N} p_{x+1}r^2_x \ \varepsilon(N) f_t^N d\mu_{1,0} \leqslant \frac{H_N(t)}{\alpha} + O(1)\ .\end{equation}
Hence,
\begin{align} \A \phi_t^N =  & \frac{\phi_t^N}{N} \sum_{x \in \T_N}  \left[\beta'_t\left(\frac{x}{N}\right)p_{x+1}r_{x}+\lambda'_t\left(\frac{x}{N}\right)p_{x+1}\right] + \frac{\phi_t^N}{N^2} \sum_{x \in \T_N} \left[\beta''_t\left(\frac{x}{N}\right)p_{x+1}r_{x}+\lambda''_t\left(\frac{x}{N}\right)p_{x+1}\right] \notag \\
& + \frac{\phi_t^N}{N} \sum_{x \in \T_N} \A(\delta_x) + o\left(\frac{1}{N}\right)\ . \end{align} 
Moreover, we can compute two equations which are called ``fluctuation-dissipation equations''. In other words, we decompose the current of  energy and the current of  deformation as the sum of a discrete gradient and a dissipative term: \begin{align}
p_{x+1} & = \nabla\left(\frac{-r_{x}}{\gamma}\right)_x+ \L^*(f_x)\ ,\label{diff1}\\
p_{x+1}r_x & = \nabla\left(-\frac{p_x^2+r_xr_{x-1}}{2\gamma}\right)_x + \L^*(g_x)\ . \label{diff2} \end{align}

We use the two equations $\eqref{diff1}$ and $\eqref{diff2}$, and we obtain  \begin{align} \A \phi_t^N =  & \frac{\phi_t^N}{N} \sum_{x \in \T_N}  \left\{\beta'_t\left(\frac{x}{N}\right)\left[\nabla\left(-\frac{p_x^2+r_xr_{x-1}}{2\gamma}\right)_x + \L^*(g_x)\right]+\lambda'_t\left(\frac{x}{N}\right)\left[\nabla\left(\frac{-r_{x}}{\gamma}\right)_x+ \L^*(f_x)\right]\right\} \notag \\
&  + \frac{\phi_t^N}{N^2} \sum_{x \in \T_N} \left[\beta''_t\left(\frac{x}{N}\right)p_{x+1}r_{x}+\lambda''_t\left(\frac{x}{N}\right)p_{x+1}\right]  + \frac{\phi_t^N}{N} \sum_{x \in \T_N} \A(\delta_x)  + o\left(\frac{1}{N}\right).\end{align} 
We sum again by part, on the two terms with a gradient, and we obtain as before 
\begin{align} \A \phi_t^N =  & \frac{\phi_t^N}{N^2} \sum_{x \in \T_N}  \left\{\beta''_t\left(\frac{x}{N}\right)\left[\frac{p_{x+1}^2+r_xr_{x+1}}{2\gamma} + p_{x+1}r_x \right]+\lambda''_t\left(\frac{x}{N}\right)\left[\frac{r_{x+1}}{\gamma}+ p_{x+1}\right]\right\} \notag \\
&  + \frac{\phi_t^N}{N} \sum_{x \in \T_N} \{\A(\delta_x)+\L^*(\delta_x) \} + o\left(\frac{1}{N}\right).\end{align} 
We get the result. \end{proof}

\begin{lem}\label{lem:sym}

\begin{equation}\S \phi_t^N=\frac{\phi_t^N}{ N}\sum_{x \in \T_N}\S(\delta_x) + \frac{\phi_t^N}{4N^2} \sum_{y \in \T_N} \left( \sum_{x \in \T_N} \delta_x(\p^y)-\delta_x(\p)\right)^2 + \phi_t^N \ \varepsilon(N)\ , \end{equation} where $\ds \mu_t^N\left[ N^2 \varepsilon(N) \right] = o(N)$\ .
\end{lem}

\begin{proof} Thanks to the exponential term, we have
\begin{equation}\S\phi_t^N=\frac{\phi_t^N}{2} \sum_{y \in \T_N} \left\{\exp\left[ \frac{1}{N}\sum_{x \in \T_N} \delta_x(\p^y)-\delta_x(\p)\right]-1\right\}\ . \end{equation}
The main idea consists in noting that $e^x-1=x+ {x^2}/{2}+o(x^2)$. We are going to give a rigorous proof of this estimate in our context thanks to the hypothesis on the energy moments. More precisely, in view of \eqref{gron} and Lemma \ref{entropy}, we want to prove that
\begin{equation} N^2 \mu_t^N\left[ \sum_{y \in \T_N} \sum_{k \geqslant 3} \frac{F_y^k}{k! \ N^k}\right] = o(N), \text{ where }  F_y = \sum_{x\in \T_N}  \left(\delta_x(\p^y)-\delta_x(\p)\right)\ . \end{equation}

Let us compute $F_y$. We notice that in the following expression, \begin{equation} \sum_{x\in \T_N} -\beta'_t\left(\frac{x}{N}\right) \frac{r_x}{2\gamma}(p_{x+1}+p_x+\frac{\gamma}{2}r_x) - \lambda'_t\left(\frac{x}{N}\right) \frac{p_{x+1}}{\gamma}\ , \end{equation} the only terms which are changing when we flip $\p$ into $\p^y$ are \begin{itemize} \item the term when $x=y$, and the difference is \begin{equation} \frac{r_yp_y}{\gamma} \beta'_t\left(\frac{y}{N}\right)\ ,\end{equation}
\item the term when $x=y-1$,  and the difference is \begin{equation}\frac{r_{y-1}p_y}{\gamma} \beta'_t\left(\frac{y-1}{N}\right)+ \lambda'_t\left(\frac{y-1}{N}\right) \frac{2 p_{y}}{\gamma}\ .\end{equation}
\end{itemize}

In other words, we have to show that \begin{equation}N \  \mu_t^N \left[ \sum_{y \in \T_N} \sum_{k \geqslant 3} \frac{ \vert F_y \vert ^k}{k! \ N^k} \right] \xrightarrow[N \to \infty]{} 0\ .\end{equation}
with \begin{align*} \vert F_y(t) \vert  & = \left\vert \frac{r_yp_y}{\gamma} \beta'_t\left(\frac{y}{N}\right)+\frac{r_{y-1}p_y}{\gamma} \beta'_t\left(\frac{y-1}{N}\right)+  \lambda'_t\left(\frac{y-1}{N}\right) \frac{2 p_{y}}{\gamma}\right\vert \\
& \leqslant C_0 \ \vert r_y p_y \vert + C_1 \ \vert r_{y-1} p_y \vert + C_2\  \vert p_y\vert \\
& \leqslant C_0 \ \frac{r_y^2 + p_y^2}{2} + C_1 \ \frac{r_{y-1}^2+p_y^2}{2} + C_2\ (1+p_y^2) \\
& \leqslant K \ (1+e_y + e_{y-1})\ , \end{align*} 
where $K$ is a constant which does not depend on $N$ and $t$.

First of all, we introduce the space $A_y=\{{e}_y \leqslant 1,\ {e}_{y-1} \leqslant 1\}$.
\begin{align} N \sum_{y \in \T_N} \mu_t^N\left[\sum_{k \geqslant 3} \frac{({e}_y+{e}_{y-1}+1)^k \ K^k\ \mathds{1}_{\{{e}_y \leqslant 1, \ {e}_{y-1} \leqslant 1\}}}{k! \ N^k}\right] & \leqslant N \sum_{y \in \T_N}\sum_{k \geqslant 3} \frac{(3K)^k}{k! \ N^k}  \notag\\
& = N^2 \ \sum_{k \geqslant 3} \frac{(3K)^k}{k! \ N^k}  \xrightarrow[N \to \infty]{} 0\end{align}

Since we have $({e}_y+{e}_{y-1})^k\mathds{1}_{A^C_y} \leqslant (2{e}_y+{e}_{y-1})^k$,  we deduce  $({e}_y+{e}_{y-1})^k\mathds{1}_{A^C_y} \leqslant C_0^k \ {e}_y^k+ C_1^k {e}_{y-1}^k\ .$
Consequently, \begin{equation} N \sum_{y \in \T_N} \mu_t^N\left[\sum_{k \geqslant 3} \frac{|F_y|^k \ K^k\ \mathds{1}_{A_y^C}}{k! \ N^k}\right] \leqslant N \sum_{y \in \T_N} \mu_t^N\left[\sum_{k \geqslant 3} \frac{{e}_y^k \ K'^k}{k! \ N^k}\right] +  N \sum_{y \in \T_N} \mu_t^N\left[\sum_{k \geqslant 3} \frac{{e}_{y-1}^k \ K'^k}{k! \ N^k}\right]. \end{equation} 
Now we deal with $N \sum_{y \in \T_N} \mu_t^N\left[\sum_{k \geqslant 3} {{e}_y^k}/({k! N^k})\right]$. Remind that $ \ds {e}_y^k \leqslant 2 \ ( p_y^{2k} +  r_y^{2k}). $

We are reduced to prove that 
 \begin{equation}N\sum_{y \in \T_N} \mu_t^N\left[\sum_{k \geqslant 3} \frac{p_y^{2k}}{k! \ N^k}\right] \xrightarrow[N \to \infty]{} 0 \quad \text{ and } \quad N\sum_{y \in \T_N} \mu_t^N\left[\sum_{k \geqslant 3} \frac{ r_y^{2k}}{k! \ N^k}\right] \xrightarrow[N \to \infty]{} 0\ . \end{equation}
We can flip the summations thanks to Fubini theorem. From the hypothesis on the moments bounds we get 
\begin{equation}N\sum_{y \in \T_N} \mu_t^N\left[\sum_{k \geqslant 3} \frac{p_y^{2k}}{k! \ N^k}\right]  \leqslant N^2\sum_{k \geqslant 3} \frac{(C\ k)^k}{k!\ N^k}  \xrightarrow[N \to \infty]{} 0\ .\end{equation}
{This last limit is deduced from the property of the series  $S(x):=\sum_{k \geqslant 3} k^k \ x^{k-2} / ({k!}) \ $. It is a power series which has a strictly positive radius and is continuous at 0. Then,}\begin{equation}N^2\sum_{k \geqslant 3} \frac{(C\ k)^k}{k!\ N^k} =C^2 S\left(\frac C N\right) \xrightarrow[N \to \infty]{} 0\ .\end{equation}The same happens for the second sum. It follows that \begin{equation}N \sum_{y \in \T_N} \mu_t^N\left[\sum_{k \geqslant 3} \frac{F_y^k}{k! \ N^k}\right] \xrightarrow[N \to \infty]{} 0\ .\end{equation} \end{proof}
After adding the two terms and get some simplifications, we obtain this following final result.

\begin{prop}\label{prop:antisym}
\begin{align} \frac{1}{\phi_t^N} N^2 \L_N^* \phi_t^N = & \sum_{x \in \T_N} \left\{-\d^2_q \beta\left(t,\frac{x}{N}\right) \left[\frac{p_{x+1}^2+r_{x+1}r_x}{2\gamma} +p_{x+1}r_x\right]- \d^2_q \lambda\left(t,\frac{x}{N}\right)\left[ \frac{r_{x+1}}{\gamma} +p_{x+1}\right] \right\} \notag \\ 
& + \frac{1}{4\gamma} \sum_{x \in \T_N} p_x^2\left[r_x \d_q \beta\left(t,\frac{x}{N}\right) + r_{x-1} \d_q \beta\left(t,\frac{x-1}{N}\right)  +2 \d_q \lambda\left(t,\frac{x-1}{N}\right)\right]^2 + o(N)\ . \label{part1} \end{align} 
\end{prop}

\begin{proof}  There are simplifications when we write  $(-\A + \gamma \S)(\phi_t^N)$. Actually,  \begin{equation}\frac{\phi_t^N}{N} \sum_{x \in \T_N} \{-\A(\delta_x)+\gamma \S(\delta_x) -\L^*(\delta_x) \} =0\ .\end{equation} The result follows. \end{proof}

\subsection{Second Term: Logarithmic Derivative }\label{ss:log}

First, we notice  that $ {\d_t \phi_t^N}/{\phi_t^N}=\d_t\{\log(\phi_t^N)\}.$ Moreover, \begin{align}
\log(\phi_t^N)=C + \sum_{x \in \T_N} & e_x\left(-\beta_t\left(\frac{x}{N}\right)+1\right)-\lambda_t\left(\frac{x}{N}\right)r_x -\beta'_t\left(\frac{x}{N}\right) \frac{r_x}{2\gamma N}(p_{x+1}+p_x+\frac{\gamma}{2}r_x) \notag \\
&  + \lambda'_t\left(\frac{x}{N}\right) \frac{p_{x}}{\gamma N} -\log\left[Z\left(\beta_t\left( \cdot \right),\lambda_t\left( \cdot \right)\right)\right]\ . \end{align}
We need to estimate the partition function $Z(\beta_t(\cdot),\lambda_t(\cdot))$. More precisely, we compare this new partition function to the exact partition function 
\begin{equation}\tilde{Z}(\beta_t(\cdot),\lambda_t(\cdot))=\prod_{x \in \T_N} \frac{2\pi}{\beta_t(x/N)} \exp\left(\frac{\lambda_t^2(x/N)}{2 \beta_t(x/N)}\right).\end{equation}
We prove the following lemma.

\begin{lem} \begin{equation}\left\vert\d_t \log Z(\beta_t(\cdot),\lambda_t(\cdot))-\d_t \log \tilde{Z}(\beta_t(\cdot),\lambda_t(\cdot))\right\vert = O(1) \quad \text{when } N \to \infty.\end{equation}
\end{lem}

\begin{proof} First of all, remind that the exact expression of $Z_t:=Z(\beta_t(\cdot),\lambda_t(\cdot))$ can be written as \begin{multline}Z_t  =\int_{\R^{2N}} \left[\prod_{x \in \T_N} \exp\left\{-\beta_t\left(\frac{x}{N}\right) e_x - \lambda_t\left(\frac{x}{N}\right) r_x \right. \right. \\ \left. \left. - \frac{1}{N} \beta'_t\left(\frac{x}{N}\right) \frac{r_x}{2\gamma}\left(p_{x+1}+p_x+\frac{\gamma}{2}r_x\right)-\frac{1}{N} \lambda'_t\left(\frac{x}{N}\right)\frac{p_{x+1}}{\gamma}\right\}\right] d\p d\r \\ = \exp\left\{\frac{1}{2} \Vert b_t \Vert ^2 \right\} \int_{\R^{2N}} \exp\left\{-\frac{1}{2} \langle X-b_t, C_t(X-b_t)\rangle\right\} dX = \exp\left\{\frac{1}{2} \Vert b_t \Vert^2 \right\} (2\pi)^N \vert \det(C_t) \vert^{1/2}. \end{multline}

where $b_t$ is a vector and $C_t$ is a symmetric positive matrix.

More precisely, one can see that \begin{equation} \Vert b_t \Vert^2= \sum_{x \in \T_N} \frac{\lambda_t^2}{\beta_t}\left(\frac{x}{N}\right) + \frac{1}{N} \sum_{x \in \T_N} h_t\left(\frac{x}{N}\right) \end{equation} where $h_t$ is a function that can be easily expressed with $\lambda_t, \beta_t, \lambda'_t$ and $\beta'_t$. Then, $h_t$ is smooth. 

Moreover, $C_t$ can be written as $C_t=D_t + N^{-1} H_t$ with $D_t$ a diagonal matrix and $H_t$ a symmetric matrix which has at most three non-zero { components} on each row and each column. More precisely, 
\begin{equation}D_t= \begin{pmatrix} 
\ddots & &  (0)  \\
 & \beta_t(x/N) & \\
(0) &  & \ddots  \end{pmatrix},\end{equation}
 \begin{equation} H_t= \begin{pmatrix} \begin{pmatrix} 
\ddots  & &  (0)  \\          
 &-(1/4)\beta'_t(x/N) &                         \\               
(0) &  & \ddots &               \end{pmatrix}       &  \begin{pmatrix}     
                                                                                      \ddots & -(2\gamma)^{-1} \beta'_t(x/N) & (0) \\
                                                                                     & -(2\gamma)^{-1} \beta'_t(x/N) & \ddots \\
                                                                                    (0) &  & \ddots \\ \end{pmatrix} \\
  \begin{pmatrix}   
   \ddots& & (0) \\                                                                               
 \ddots &   -(2\gamma)^{-1} \beta'_t(x/N) & \\                                                                                 
    (0) &   -(2\gamma)^{-1} \beta'_t(x/N) & \ddots \end{pmatrix} & \begin{pmatrix}      0  \end{pmatrix}    
                                                                                \end{pmatrix}
                                                                                    \end{equation}

Now we write \begin{align} \d_t \log Z_t & = \frac{1}{2} \sum_x \d_t \left(\frac{\lambda_t^2}{\beta_t}\left(\frac{x}{N} \right)\right) + \frac{1}{2} \d_t \log \det(C_t) + \frac{1}{N} \sum_x \d_t h_t\left(\frac{x}{N} \right), \\
\d_t \log \tilde{Z}_t & = \frac{1}{2} \sum_x \d_t \left(\frac{\lambda_t^2}{\beta_t}\left(\frac{x}{N} \right)\right) + \frac{1}{2} \d_t \log \det(D_t)\ . \end{align}
But, $ \left\vert {N}^{-2} \sum_x \d_t h_t\left(x/N \right) \right\vert = O(1)$ since $h_t$ is smooth.

It remains to show that the following quantity is bounded above by a constant that does not depend on $N$: 
\begin{equation}\left\vert \d_t \left(\log \frac{\det C_t}{\det D_t}\right) \right\vert = \left\vert \d_t \left[\log \det\left(I + \frac{1}{N} D_t^{-1} H_t\right) \right]\right\vert= \left\vert \frac{\d_t \{ \det (I+ D_t^{-1} H_t /N) \}}{\det (I+D_t^{-1} H_t /N)} \right\vert\ .\end{equation}
We denote by $K_t$ the matrix $D_t^{-1} H_t$, which also has at most three non-zero { components} on each row and each column, and by $K'_t$ the derivative of $K_t$ with respect to $t$. We notice that for $N$ { large} enough, the matrix $I + K_t/N$ is invertible, and we have \begin{equation}\left\vert \d_t \left(\log \frac{\det C_t}{\det D_t}\right) \right\vert = \left\vert\frac{\text{Tr}(\  \!^t\text{com}(I+K_t/N) \cdot (I+K'_t/N))}{\det (I+K_t/N)} \right\vert = \left\vert \text{Tr}\left[ \left(I+\frac{1}{N} K_t\right)^{-1} \left(I+K'_t\right) \right]  \right\vert\ ,\end{equation} where $\text{com}(A)$ is the comatrix of $A$.

Now we deal with $(I+K_t/N)^{-1}$: \begin{equation}\left(I+\frac{1}{N} K_t\right)^{-1} = I - K_t + \sum_{k \geqslant 2} \frac{(-1)^k}{N^k} K_t^k\ . \end{equation}
But, the {component} $(i,j)$ of $K_t^k$ can be written as $ \sum_{i_1,...,i_k} a_{i,i_1} \ a_{i_1,i_2} \cdots\ a_{i_k,j}$ where $a_{i,j}$ are the { components} of $K_t$. We know that there are at most three non-zero { components} on each row and each column, and that they are all bounded by a constant $C$ that does not depend on $N$ (since $\beta_t$ and $\lambda_t$ are smooth). Then, it implies that $\vert \text{Tr}(K_t^k) \vert \leqslant N 3^k C.$

It follows that \begin{equation}\left\vert \text{Tr}\left[ \left(I+\frac{1}{N}K_t\right)^{-1}\right]\right\vert = \left\vert \text{Tr}\left( I-K_t + \sum_{k \geqslant 2} \frac{(-1)^k}{N^k} K_t^k \right)\right\vert \leqslant 1 + \vert \text{Tr}(K_t) \vert +C  \sum_{k \geqslant 2} \frac{3^k}{N^{k-1}} = O(1)\ ,\end{equation}
because $\text{Tr}(K_t) =O(1)$ (we can compute it and again use the smoothness of the profiles).

In the same way, we show that \begin{equation}\left\vert \text{Tr}\left[ K'_t  \left(I+\frac{1}{N}K_t\right)^{-1}\right]\right\vert =O(1)\ . \end{equation} It ends the proof. \end{proof}

We deduce from the previous result that \begin{equation} \d_t \log\left[Z\left(\beta_t\left(\cdot \right),\lambda_t\left(\cdot\right)\right)\right]  = \sum_{x \in \T_N} -\frac{\d_t \beta_t(x/N)}{\beta_t(x/N)}+\d_t \lambda_t(x/N) \frac{\lambda_t(x/N)}{\beta_t(x/N)}-\frac{\d_t \beta_t(x/N)}{2} \frac{\lambda_t^2(x/N)}{\beta_t^2(x/N)} + O(1)\ .\end{equation}
Consequently, we get the following statement.
 
\begin{prop}\label{prop:z}
\begin{align}
\d_t\{\log(\phi_t^N)\} = \sum_{x \in \T_N} & - e_x \d_t \beta\left(t,\frac{x}{N}\right) - r_x \d_t \lambda\left(t,\frac{x}{N}\right) - \frac{r_x}{2 \gamma N} \d_t \d_q \beta\left(t,\frac{x}{N}\right)\left(p_{x+1}+p_x+\frac{\gamma}{2}r_x\right) \notag \\
& - \frac{p_{x}}{\gamma N} \d_t \d_q \lambda\left(t,\frac{x}{N}\right) + \frac{\d_t \beta(t,x/N)}{\beta(t,x/N)} \notag \\
 & -\d_t \lambda(t,x/N) \frac{\lambda(t,x/N)}{\beta(t,x/N)}+\frac{\d_t \beta(t,x/N)}{2} \frac{\lambda^2(t,x/N)}{\beta^2(t,x/N)}, \\
 \notag \\
\d_t\{\log(\phi_t^N)\}  = \sum_{x \in \T_N} & -\left[e_x-\mathbf{e}\left(t,\frac{x}{N}\right)\right] \d_t \beta\left(t,\frac{x}{N}\right) + \left[r_x -\mathbf{r}\left(t,\frac{x}{N}\right)\right] \d_t \lambda\left(t,\frac{x}{N}\right) + O(1)\ . \end{align}
\end{prop}

\subsection{Ending Proof}\label{subsec:end}

We are now able to prove the Taylor expansion. According to the results of the two previous parts, we have
\begin{align} \frac{1}{\phi_t^N} N^2 \L_N^* \phi_t^N -\d_t\{\log(\phi_t^N)\}  = & \sum_{x \in \T_N} \left\{-\d^2_q \beta\left(t,\frac{x}{N}\right) \left[\frac{p_{x}^2+r_{x-1}r_x}{2\gamma} +p_{x}r_{x-1}\right]-\d^2_q \lambda\left(t,\frac{x}{N}\right)\left[ \frac{r_{x}}{\gamma} +p_{x}\right] \right. \notag \\ 
& + \frac{p_x^2}{4\gamma}\left[(r_x+r_{x-1}) \d_q \beta\left(t,\frac{x}{N}\right) +2 \d_q \lambda\left(t,\frac{x}{N}\right)\right]^2 \notag \\
 & \left.+ \left[e_x-\mathbf{e}\left(t,\frac{x}{N}\right)\right] \d_t \beta\left(t,\frac{x}{N}\right) + \left[r_x -\mathbf{r}\left(t,\frac{x}{N}\right)\right] \d_t \lambda\left(t,\frac{x}{N}\right)\right\} + o(N) \ . \label{part1} \end{align} 

Using the notations introduced in Section \ref{sec:entropy}, it becomes: 
\begin{align} \frac{1}{\phi_t^N} N^2 \L_N^* \phi_t^N -\d_t\{\log(\phi_t^N)\}  = & \sum_{x \in \T_N} \left\{-\frac{1}{2\gamma} \d^2_q \beta\left(t,\frac{x}{N}\right) J_x^1-\frac{1}{\gamma} \d^2_q \lambda\left(t,\frac{x}{N}\right)  J_x^2 \right. \notag \\ 
& + \frac{1}{4\gamma}\left[\d_q \beta\left(t,\frac{x}{N}\right)\right]^2  J_x^3 + \frac{1}{\gamma}  \d_q \beta\left(t,\frac{x}{N}\right) \d_q \lambda\left(t,\frac{x}{N}\right) J_x^4 \notag \\
& +  \frac{1}{\gamma} \left[\d_q \lambda\left(t,\frac{x}{N}\right)\right]^2   J_x^5     \notag \\
 & \left.+ \left[e_x-\mathbf{e}\left(t,\frac{x}{N}\right)\right] \d_t \beta\left(t,\frac{x}{N}\right) + \left[r_x -\mathbf{r}\left(t,\frac{x}{N}\right)\right] \d_t \lambda\left(t,\frac{x}{N}\right)\right\} + o(N)\ . \end{align} 

We denote by $H_k$ the function defined as follows: \begin{equation}H_k\left(\eta\left(t,\frac{x}{N}\right)\right)=\mu^N_{\chi_t(x/N)}\left[J_0^k\right]. \end{equation} The explicit formulations for $H_k$ are given by Proposition \ref{prop}. The sum \begin{align}\sum_{x \in \T_N} & \left\{-\frac{1}{2\gamma} \right.  \d^2_q \beta\left(t,\frac{x}{N}\right)  H_1\left(\eta\left(t,\frac{x}{N}\right)\right)-\frac{1}{\gamma} \d^2_q \lambda\left(t,\frac{x}{N}\right)  H_2\left(\eta\left(t,\frac{x}{N}\right)\right) \notag \\
& + \frac{1}{4\gamma}\left[\d_q \beta\left(t,\frac{x}{N}\right)\right]^2  H_3\left(\eta\left(t,\frac{x}{N}\right)\right) + \frac{1}{\gamma}  \d_q \beta\left(t,\frac{x}{N}\right) \d_q \lambda\left(t,\frac{x}{N}\right)  H_4\left(\eta\left(t,\frac{x}{N}\right)\right) \notag \\
&  \left.  +  \frac{1}{\gamma} \left[\d_q \lambda\left(t,\frac{x}{N}\right)\right]^2   H_5\left(\eta\left(t,\frac{x}{N}\right)\right)\right\}   \end{align} is of order $o(N)$ (thanks to the regularity of the functions $\mathbf{e}, \mathbf{r}, \beta, \lambda$), so that we can introduce it in the right member of the equality \eqref{part1}.

Then, we obtain after computations
\begin{equation}-\frac{ \d^2_q \beta}{2\gamma} \ \d_{\mathbf{e}} H_1 - \frac{ \d^2_q \lambda}{\gamma} \ \d_{\mathbf{e}} H_2  + \frac{ \left[\d_q\beta \right]^2}{4\gamma} \ \d_{\mathbf{e}} H_3 + \frac{ \d_q\beta \d_q\lambda}{\gamma} \  \d_{\mathbf{e}} H_4  + \frac{\left[\d_q\lambda \right]^2}{\gamma}  \\d_{\mathbf{e}} H_5  = -\d_t \beta\ ,\end{equation} 
and
\begin{equation}-\frac{ \d^2_q \beta}{2\gamma} \  \d_{\mathbf{r}} H_1 - \frac{ \d^2_q \lambda}{\gamma} \ \d_{\mathbf{r}} H_2 + \frac{ \left[\d_q\beta \right]^2}{4\gamma} \ \d_{\mathbf{r}} H_3 + \frac{ \d_q\beta \d_q\lambda}{\gamma} \  \d_{\mathbf{r}} H_4  + \frac{\left[\d_q\lambda \right]^2}{\gamma}  \ \d_{\mathbf{r}} H_5  = \d_t \lambda\ .\end{equation}
Indeed, these two quantities are respectively equal to
\begin{equation}\frac{ \d^2_q \beta}{2\gamma}  - \frac{ [\d_q \beta]^2}{\gamma} \ \left(\mathbf{e}+\frac{\mathbf{r}^2}{2}\right)  - 2\mathbf{r} \ \frac{ \d_q\beta \d_q\lambda}{\gamma}  - \frac{\left[\d_q\lambda \right]^2}{\gamma}\ , \end{equation}
and
\begin{equation} \frac{ \d^2_q \beta}{2\gamma} \ \mathbf{r}+ \frac{ \d^2_q \lambda}{\gamma}- \frac{ \left[\d_q\beta \right]^2}{2\gamma} \ \mathbf{r} \ (2\mathbf{e}-3\mathbf{r}^2) - \frac{ \d_q\beta \d_q\lambda}{\gamma} \ (2\mathbf{e}-3\mathbf{r}^2) + \mathbf{r }\frac{\left[\d_q\lambda \right]^2}{\gamma}\ .  \end{equation}
This concludes the proof and gives Proposition \ref{prop}.

\section{Proof of the One-block Estimate}\label{appb}

We just give a sketch of the proof, which is done in \cite{BOmanus}, { Section 3.4}. First, we define the space time average of distribution: \begin{equation}\bar{f}^N = \frac{1}{tN} \sum_{i=1}^N \int_0^t \tau_i f_s^N ds\ ,\end{equation}
and $\bar{f}_k^N$ its projection on $\{ (r_i,p_i) \in \R^{2(k+1)} \ ; \ i \in \Lambda_k:= \{ -[k/2]-1,…, [k/2]+1\} \}$.

We also denote $ d\nu^N=\bar{f}^N \prod_{i \in \T_N} dr_idp_i$ and $d\nu_k^N=\bar{f_k}^N \prod_{i \in \T_N} dr_idp_i $ the corresponding probability measures on $\R^{2N}$ and $\R^{2(k+1)}$.

Observe first that \eqref{oneb} can be rewritten as  \begin{equation}t \limsup_{M \to \infty} \limsup_{\ell \to \infty} \limsup_{N \to \infty} \int \left\{ \left\vert \frac{1}{\ell} \sum_{i \in \Lambda_{\ell}(0)} J_{i,M}-H({\eta}_{\ell,M}(0))\right\vert \right\} d\nu^N=0\ , \end{equation} 
 because \begin{equation}\frac{1}{\ell} \sum_{k=0}^{\ell-1} \frac{1}{p} \sum_{j=1}^p\tau_{x_j +k} = \frac{1}{N} \sum_{x=1}^N \tau_x.\end{equation}
  We can prove the first following lemma.
 
 \begin{lem} \label{lem:tight} For each fixed $k$, the sequence of probability measures $(\nu_k^N)_{N \geqslant k}$ is tight. \end{lem}
 
 For any $k$ let $\nu_k$ be a limit point of the sequence $(\nu_k^N)_{N \geqslant 1}$. The sequence of probability measures $(\nu_k)_{k \geqslant 1}$ forms a consistent family and by Kolmogorov's theorem there exists a unique probability measure $\nu$ on $(\R \times \R)^{\Z}$ such that the restriction of $\nu$ on $\{ (r_i,p_i) \in \R^{2(k+1)} \ ; \ i \in \Lambda_k\}$ is $\nu_k$. One has easily that $\nu$ is invariant by translations. 
 
 \begin{lem} \label{lem:invariant} For any bounded smooth local function $F(\r,\p)$, we have $\ds \int \L F d\nu =0.$
 
 \end{lem} Then, $\nu$ is a convex combination of grand canonical Gibbs measures $\mu_{\chi}=\mu_{\beta,\lambda}$: $ \nu=\int d\rho(\chi) \mu_{\chi}$, with $\rho$ a probability measure such that $ \int d\rho(\chi) \mu_{\chi}[e_j] \leqslant C_0$ for any $j \in \Z$. 
 
 Hence, it results that  \begin{align} \limsup_{M \to \infty} & \limsup_{\ell \to \infty} \limsup_{N \to \infty} \int  \left\{ \left\vert \frac{1}{\ell} \sum_{i \in \Lambda_{\ell}(0)} J_{i,M}-H({\eta}_{\ell,M}(0))\right\vert \right\} d\nu^N \notag \\ & =\limsup_{M\to \infty} \limsup_{\ell \to \infty} \int d\rho(\chi) \int  \left\{ \left\vert \frac{1}{\ell} \sum_{i \in \Lambda_{\ell}(0)} J_{i,M}-H({\eta}_{\ell,M}(0))\right\vert \right\} d\mu_{\chi} \notag \\ 
 &  = \limsup_{M\to \infty} \int d\rho(\chi) \left[\limsup_{\ell \to \infty} \int  \left\{ \left\vert \frac{1}{\ell} \sum_{i \in \Lambda_{\ell}(0)} J_{i,M}-H({\eta}_{\ell,M}(0))\right\vert \right\} d\mu_{\chi}\right],\end{align}
 where the last equality is a consequence of the {dominated convergence theorem. }
Since $\mu_{\chi}$ is ergodic with respect to $\{\tau_x \ ;\ x \in \Z\}$, the last term is equal to \begin{equation}\limsup_{M\to\infty} \int d\rho(\chi) \left\vert \mu_{\chi} \left[J_{0,M}\right]-H(\mu_{\chi}[\eta_{0,M}])\right\vert.\end{equation}
 As $M\to \infty$, $\mu_{\chi}[J_{0,M}]$ converges to $\mu_{\chi}[J_0]=H\left(\mu_{\chi}[\xi_0]\right)$ and $\mu_{\chi}[\xi_{0,M}]$ to $\mu_{\chi}[\xi_0]$.
 
 By Fatou's lemma, the limit in $M$ is equal to 0 and this concludes the proof of the one-block lemma.

\newpage 

\providecommand{\bysame}{\leavevmode\hbox to3em{\hrulefill}\thinspace}
\providecommand{\href}[2]{#2}


\begin{thebibliography}{Sim}


\bibitem{BBO}
G.~Basile, C.~Bernardin, and S.~Olla, \emph{Thermal conductivity for a momentum conservative model}, Comm. Math. Phys. {\bf{287}} (2009), no. 1, 67--98.


\bibitem{MR2305383}
C.~Bernardin, \emph{Hydrodynamics for a system of harmonic oscillators
  perturbed by a conservative noise}, Stochastic Process. Appl. \textbf{117}
  (2007), no.~4, 487--513. 

\bibitem{MR2185330}
C.~Bernardin and S.~Olla, \emph{Fourier's law for a microscopic model of heat
  conduction}, J. Stat. Phys. \textbf{121} (2005), no.~3-4, 271--289.

\bibitem{BOmanus}
\bysame, \emph{Non-equilibrium macroscopic dynamics of chains
 { of} anharmonic oscillators}, in preparation, available at
  http://www.ceremade.dauphine.fr/olla (2011).

\bibitem{BerOlla}
\bysame, \emph{Transport properties of a chain of anharmonic oscillators with
  random flip of velocities}, J. Stat. Phys \textbf{145} (2011), 1124--1255.

\bibitem{BS12} C.~Bernardin and G.~Stoltz, \emph{Anomalous diffusion for a class of systems with two conserved quantities}, Nonlinearity {\bf 25} (2012), 1099--1133.

\bibitem{BerKan}
C.~Bernardin, V.~Kannan, J.~L. Lebowitz, and J.~Lukkarinen, \emph{Harmonic
  systems with bulk noises}, Eur. Phys. J. B. \textbf{84} (2011), 685--689. 

\bibitem{MR2192537}
L.~Bertini, D.~Gabrielli, and J.~L. Lebowitz, \emph{Large deviations for a
  stochastic model of heat flow}, J. Stat. Phys. \textbf{121} (2005), no.~5-6,
  843--885. 



\bibitem{MR2518984}
F.~Bonetto, J.~L. Lebowitz, J.~Lukkarinen, and S.~Olla, \emph{Heat conduction
  and entropy production in anharmonic crystals with self-consistent stochastic
  reservoirs}, J. Stat. Phys. \textbf{134} (2009), no.~5-6, 1097--1119.





\bibitem{EveOll}
N.~Even and S.~Olla, \emph{Hydrodynamic limit for an hamiltonian system with
  boundary conditions and conservative noise}, arXiv:1009.2175v1 (2011).

\bibitem{MR1278883}
J.~Fritz, T.~Funaki, and J.~L. Lebowitz, \emph{Stationary states of random
  {H}amiltonian systems}, Probab. Theory Related Fields \textbf{99} (1994),
  no.~2, 211--236. 

\bibitem{MR1395890}
T.~Funaki, K.~Uchiyama, and H.~T. Yau, \emph{Hydrodynamic limit for lattice gas
  reversible under {B}ernoulli measures}, Nonlinear stochastic {PDE}s
  ({M}inneapolis, {MN}, 1994), IMA Vol. Math. Appl. \textbf{77} (1996), 1--40. 

\bibitem{MR1707314}
C.~Kipnis and C.~Landim, \emph{Scaling limits of interacting particle systems},
  Grundlehren der Mathematischen Wissenschaften [Fundamental Principles of
  Mathematical Sciences], \textbf{320} (1999), Springer-Verlag, Berlin. 

\bibitem{MR2080952}
C.~Landim, M.~Sued, and G.~Valle, \emph{Hydrodynamic limit of asymmetric
  exclusion processes under diffusive scaling in {$d\geq 3$}}, Comm. Math.
  Phys. \textbf{249} (2004), no.~2, 215--247. 

\bibitem{LO} C.~Liverani, and S.~Olla, \emph{Toward the Fourier law for a weakly interacting anharmonic crystal}, J. Amer. Math. Soc. {\bf 25} (2012), no. 2, 555--583.

\bibitem{MR1231642}
S.~Olla, S.~R.~S. Varadhan, and H.~T. Yau, \emph{Hydrodynamical limit for a
  {H}amiltonian system with weak noise}, Comm. Math. Phys. \textbf{155} (1993),
  no.~3, 523--560. 

\bibitem{MR1934159}
C.~Tremoulet, \emph{Hydrodynamic limit for interacting
  {O}rnstein-{U}hlenbeck particles}, Stochastic Process. Appl. \textbf{102}
  (2002), no.~1, 139--158. 

\bibitem{MR1354152}
S.~R.~S. Varadhan, \emph{Nonlinear diffusion limit for a system with nearest
  neighbor interactions. {II}}, Asymptotic problems in probability theory:
  stochastic models and diffusions on fractals ({S}anda/{K}yoto, 1990), Pitman
  Res. Notes Math. Ser., \textbf{283} (1993), 75--128. 

\bibitem{MR1121850}
H.~T. Yau, \emph{Relative entropy and hydrodynamics of {G}inzburg-{L}andau
  models}, Lett. Math. Phys. \textbf{22} (1991), no.~1, 63--80. 


\end{thebibliography}
\end{document}